\documentclass{siamart0216}

\usepackage{amsfonts}
\usepackage{amssymb}
\usepackage{algorithm}
\usepackage{algorithmic}
\usepackage{eqparbox}

\usepackage{etoolbox}  % patch def of algorithmic environment
\makeatletter
\patchcmd{\algorithmic}{\addtolength{\ALC@tlm}{\leftmargin} }{\addtolength{\ALC@tlm}{\leftmargin}}{}{}
\makeatother
\usepackage{enumerate}
\usepackage{subfigure}
\usepackage{courier}
\usepackage{setspace}
\usepackage{multirow}
\usepackage{booktabs}
\usepackage[amsmath,thmmarks]{ntheorem}
\usepackage{soul}
\usepackage{booktabs}
\usepackage{threeparttable}

\newtheorem{Mytheorem}{theorem}
\newtheorem{MyProp}{proposition}
\newcommand{\bQ}{{\boldsymbol{Q}}}

\newcommand{\bI}{{\boldsymbol{I}}}

\newcommand{\bR}{{\boldsymbol{R}}}
\newcommand{\bB}{{\boldsymbol{B}}}
\newcommand{\bA}{{\boldsymbol{A}}}

\newcommand{\bG}{{\boldsymbol{G}}}
\newcommand{\bH}{{\boldsymbol{H}}}
\newcommand{\bY}{{\boldsymbol{Y}}}

\newcommand{\nT}{{T}}
\newcommand{\nF}{{\textnormal{F}}}
\newcommand{\randQBb}{\texttt{randQB\_b}}
\newcommand{\qrcp}{{\texttt{trQRCP}}}
\newcommand{\randQB}{\texttt{randQB}}
\newcommand{\randQBfp}{\texttt{randQB\_FP}}
\newcommand{\randQBei}{\texttt{randQB\_EI}}

\renewcommand{\st}[1]{}

\title{Efficient Randomized algorithms for the fixed-precision low-rank matrix approximation
}
\author{Wenjian Yu\thanks{TNList, Department of Computer Science and Technology, Tsinghua University, Beijing 100084, China ({\tt yu-wj@tsinghua.edu.cn}).}
    \and Yu Gu\thanks{Department of Computer Science and Technology and Institute for Interdisciplinary Information Sciences, Tsinghua University, Beijing 100084, China ({\tt guyu13@mails.tsinghua.edu.cn}).}
\and Yaohang Li\thanks{Department of Computer Science, Old Dominion University, Norfolk, VA 23529, USA ({\tt yaohang@cs.odu.edu}).} 
}

\begin{document}

\maketitle

\begin{abstract}
Randomized algorithms for low-rank matrix approximation are investigated, with the emphasis on the fixed-precision problem and computational efficiency for handling large matrices. The algorithms are based on the so-called QB factorization, where $\bQ$ is an orthonormal matrix. 
Firstly, a mechanism for calculating the approximation error in \emph{Frobenius norm} is proposed, which enables efficient adaptive rank determination for large and/or sparse matrix. It can be combined with any QB-form factorization algorithm in which $\bB$'s rows are incrementally generated. Based on the blocked randQB algorithm by P.-G. Martinsson and S. Voronin, this results in an algorithm called \texttt{randQB\_EI}.
Then, we further revise the algorithm to obtain a pass-efficient algorithm, \texttt{randQB\_FP}, which is mathematically equivalent to the existing randQB algorithms and also suitable for the fixed-precision problem. Especially, \texttt{randQB\_FP} can serve as a single-pass algorithm for calculating leading singular values, under certain condition. With large and/or sparse test matrices, we have empirically validated the merits of the proposed techniques, which exhibit remarkable speedup and memory saving over the blocked randQB algorithm. We have also demonstrated that the  single-pass algorithm derived by {\randQBfp} is much more accurate than an existing single-pass algorithm. And with data from a scenic image and an information retrieval application, we have shown the advantages of the proposed algorithms over the adaptive range finder algorithm for solving the fixed-precision problem.  

\end{abstract}



\section{Introduction}

Low-rank matrix factorizations, like the partial singular value decomposition (SVD) and the rank-revealing QR factorization, play a crucial role in data analysis and scientific computing. In recent years, techniques based on randomization have been investigated for performing the computation and low-rank factorization of large matrices \cite{Martin2015, Halko2011, cacm16, Voronin2015,
%Martin2006, Martin2011, 
EB2011, Drineas2006, Rokhlin2009, YaohangLi, sc15, Mahoney11}.  They involve the same or fewer floating-point operations ($flops$) than classical algorithms, and are more efficient by exploiting modern computing architectures. 

A basic idea of the randomized techniques is using random projection to approximate the dominant subspace  of a matrix. For an $m\times n$ matrix $\boldsymbol{A}$, suppose the orthogonal basis vectors  of this approximate subspace form an $m\times k$ orthonormal matrix $\boldsymbol{Q}$. Then,
%$k$-dimensional dominant subspace has a set of orthogonal basis vectors forming a matrix $\boldsymbol{Q}$ (here $k<\min (m,n)$). Because the range space of $\bA$ is well approximated by the span of $\bQ$'s columns, 
we have \cite{Martin2015, Halko2011}:
\begin{equation}\label{eq:qb}
\boldsymbol{A} \approx \boldsymbol{QB},
\end{equation}
where $\boldsymbol{B}$ is a $k\times n$ matrix, and
\begin{equation}\label{eq:Bformula1}
	\boldsymbol{B} = \boldsymbol{Q}^\nT \boldsymbol{A}.
\end{equation}
Standard factorizations, e.g., SVD, can be further performed on the smaller matrix $\boldsymbol{B}$, to obtain the low-rank factorizations of  $\boldsymbol{A}$. 

The approximation presented by (\ref{eq:qb}) and (\ref{eq:Bformula1}) can also be regarded as a kind of low-rank factorization of $\bA$, called \emph{QB factorization} or \emph{QB approximation} in this work. In [2], a basic randomized scheme for computing the QB approximation was presented, as shown in Figure 1(a). For producing close to optimal rank-$k$ approximation, the over-sampling scheme using a random Gaussian matrix $\boldsymbol{\Omega}$ with $k+s$ columns is  employed, where $s$ is a small integer.
% ``orth($\boldsymbol{X}$)'' in Figure 1(a) denotes orthonormalization of the columns of $\boldsymbol{X}$. In practice, it is realized efficiently by a call to a packaged QR factorization. 
We use {\randQB} to denote this algorithm. 
%More other randomized schemes for computing low-rank matrix factorization and least-squares approximation can be found in \cite{cacm16,Mahoney11}.

\begin{figure*}[]
\centering
\subfigure[The \texttt{randQB} algorithm] {\includegraphics[height=1in]{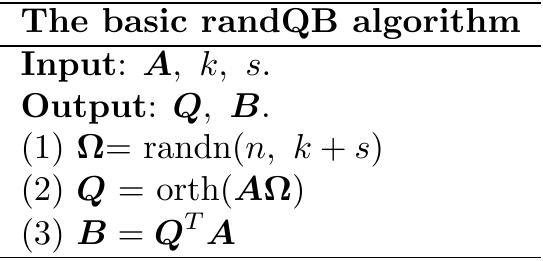}}
\subfigure[The \texttt{randQB\_b} algorithm] {\includegraphics[width=2.58in]{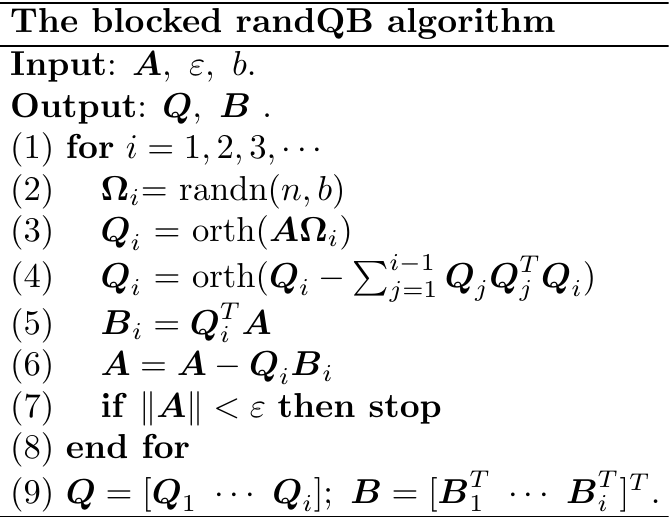}}
\caption{The randomized algorithms for the QB factorization.}
\label {fig1}
\end{figure*}

Usually, the problem of low-rank matrix approximation falls into two categories:
\begin{itemize}
\item The \emph{fixed-rank} problem, where the  rank parameter $k$  is given.
\item The \emph{fixed-precision} problem, where we seek $\bQ$ and $\boldsymbol{B}$ with as small as possible size such that $\| \boldsymbol{A} - \boldsymbol{QB} \| < \varepsilon$, where $\varepsilon$ is a given accuracy tolerance. 
%More practically, the relative accuracy tolerance is often considered, where $\varepsilon$ is a fraction of $\|\bA\|$, i.e., $\epsilon\|\bA\|$.
\end{itemize}
% By the Eckart-Young theorem \cite{E-Y1936}, the truncated SVD provides the optimal low-rank approximation. In contrast, techniques like \texttt{randQB} and partial pivoted QR do not produce the optimal approximation. However, computing SVD of a large matrix is costly, and in many applications the optimality is not necessary. It is acceptable to trade some optimality in accuracy for gains in computational efficiency.

A blocked variant of the {\randQB} algorithm proposed in \cite{Martin2015}, i.e. the {\randQBb} algorithm in Figure 1(b), is suitable for the fixed-precision problem. It incrementally builds the factors $\bQ$ and $\bB$ based on the combination of the {\randQB} algorithm and the \emph{blocked} Gram-Schmidt scheme, and  measures the approximation error by explicitly maintaining the residual matrix. However, it is inefficient or even fails for handling large matrix, because maintaining the residual matrix is costly in runtime and memory usage.

%an adaptive randomized range finder was proposed in \cite{Halko2011}. It employs the incremental sampling approach with a probabilistic error estimator to determine the size of $\boldsymbol{Q}$ and $\boldsymbol{B}$. However, the error estimator often overestimates the approximation error, yielding much larger output matrices than what is necessary. In , Martinsson and Voronin proposed a randomized blocked algorithm (i.e., \randQBb) for computing rank-revealing factorizations. It incrementally builds the factors $\bQ$ and $\bB$ based on the combination of the basic {\randQB} procedure and the \emph{blocked} Gram-Schmidt scheme. It measures the approximation error by explicitly maintaining the residual matrix, instead of using the probabilistic error estimator. This makes it suitable for the fixed-precision problem. However, it is inefficient or even fails for handling large matrix, because maintaining the residual matrix is costly in runtime and memory usage. 

In this work, the randomized algorithms for the fixed-precision problem are investigated considering their adaptability to large and/or sparse matrices. 
Firstly, a mechanism is proposed for calculating the error of QB approximation in \emph{Frobenius norm} during the iterative process of building $\boldsymbol{Q}$ and $\boldsymbol{B}$. It does not require maintaining the residual matrix (or updating matrix $\bA$), and thus avoids fill-in while handling a sparse $\bA$.
This mechanism is also applicable to other iterative computing procedures, e.g., the {\qrcp} algorithm \cite{Duersch2015}. Secondly, the algorithm is further revised  to largely reduce the number of passes over matrix $\bA$, in order to adapt the scenarios where the cost of accessing matrix $\bA$ is expensive.
These techniques result in two algorithms called \texttt{randQB\_EI} and \texttt{randQB\_FP}, which inherit the merits of {\randQB}/{\randQBb} algorithms and have extra benefits.
Numerical experiments are carried out on a multi-core computer to validate the efficiency and accuracy of the proposed algorithms for handling large or sparse matrices in practical scenarios. The results show that our \texttt{randQB\_EI} and \texttt{randQB\_FP} algorithms have up to 3X speedup and 3X memory saving over an implementation of the {\randQBb} algorithm for dense matrices. They also exhibit up to 22X speedup over the implementation of {\randQBb} algorithm for  sparse matrices.
Compared with the single-pass algorithm and the adaptive randomized range finder  in \cite{Halko2011}, the proposed algorithms exhibit much better accuracy or avoid the  large overestimation of the sizes of $\bQ$ and $\bB$. For reproducibility, we have shared the Matlab codes of
the proposed algorithms and experimental data on \url{https://github.com/WenjianYu/randQB_auto}.

\section{Technical preliminaries}
This section summarizes the background we need for presenting the proposed techniques.
Throughout the paper, we measure vectors with their Euclidean norm. Two kinds of matrix norm are usually considered: Frobenius norm and spectral norm ($l_2$-norm). The spectral norm of a matrix is relatively difficult to calculate, though it is often more informative for noisy data \cite{Tygert2017}. The Frobenius norm $\|\boldsymbol{A}\|_{\nF}=(\sum_{i,j}|\boldsymbol{A}(i,j)|^2)^{1/2}$ is easier for calculation, and thus more widely used in data analysis and machine learning applications \cite{Mahoney11}. We measure matrices with their Frobenius norm by default. We also assume that all matrices are real valued, although the generalization to complex matrices is of no difficulty. 

\subsection{Randomized algorithms}
% The approach based on the partial QR factorization can be employed to produce a close to optimal low-rank approximation, but the randomized technique recently developed can provide a more efficient solution. It produces good rank-$k$ factorizations costing $O(mnk)$ or fewer flops, and also has good scalability on multi-core computing platforms. 

To produce a rank-$k$ factorization for an $m \times n$ matrix $\boldsymbol{A}$, 
with the basic randQB algorithm in Figure 1(a) one obtains an $m\times l$  orthonormal matrix $\boldsymbol{Q}$ and an $l\times n$ matrix $\boldsymbol{B}$, where $l>k$ due to over-sampling. %Then, the $l\times n$ matrix $\boldsymbol{B}$ is computed from $\boldsymbol{B} = \boldsymbol{Q}^T \boldsymbol{A}$, to minimize the approximation error.
With this QB approximation, the standard factorizations can be efficiently computed. For example, the standard SVD algorithm can be performed on $\boldsymbol{B}$, which results in $\boldsymbol{B}= \tilde{\boldsymbol{U}}\tilde{\boldsymbol{\Sigma}}\tilde{\boldsymbol{V}}^T$. Then,
\begin{equation}\label{eq:approxSVD}
\boldsymbol{A} \approx \boldsymbol{QB} = \boldsymbol{Q}\tilde{\boldsymbol{U}}\tilde{\boldsymbol{\Sigma}}\tilde{\boldsymbol{V}}^T.
\end{equation}
The first $k$ columns of  matrices $\boldsymbol{Q}\tilde{\boldsymbol{U}}$ and $\tilde{\boldsymbol{V}}$ and the 
$k\times k$ upper-left submatrix of $\tilde{\boldsymbol{\Sigma}}$ approximate the rank-$k$ SVD factors of $\bA$. Similarly, by changing the factorizations made on $\boldsymbol{B}$, one obtains the approximate QR factorization and CUR factorization, etc \cite{Martin2015, Voronin2015}. Notice that the accurate truncated SVD provides the optimal low-rank approximation \cite{E-Y1936}. However, computing accurate SVD of a large matrix is costly, and in many applications the optimality is not necessary. It is thus acceptable to compute the approximate low-rank factorizations for gains in computational efficiency.

% The key step for finding the matrix $\boldsymbol{Q}$  is multiplying $\boldsymbol{A}$ with a random matrix, such as $\boldsymbol{\Omega}$ in the procedure shown in Figure 1. $\boldsymbol{\Omega}$ is usually chosen as a standard Gaussian matrix; i.e., its entries are independent standard normal variables of zero mean and standard deviation 1. Although other random matrices might work equally well, the choice of the Gaussian matrix provides some theoretical and practical advantages \cite{Halko2011, Gu2015}. 

%The output of the randomized approximation algorithms is a random variable, as it depends on the drawing of a Gaussian matrix. It has been proven that the variation in this random variable is small, which means the output is always very close to the variable's expectation. For more details, please refer to [1, 2]. The bound of the expectation of the approximation error has been derived in [2]. And, the distribution of errors caused by randomness has been demonstrated in [1].
The approximation error of the randomized algorithm is a random variable. The authors of [2] have studied the properties of the error in term of spectral and Frobenius norms, and given the bounds on their expectation and variance.

%\footnote{Here, we consider the sparse matrix whose nonzero entries are explicitly represented, which widely exists in information retrieval problems or for representing network or graph associated data.} %The basic randomized scheme may be further accelerated by utilizing a \emph{structured} random matrix to substitute for $\boldsymbol{\Omega}$, or by the sparsity-preserving random projection technique while handling a sparse $\bA$ \cite{Halko2011, Mahoney11, sparseJL}. In this work we only consider the basic algorithm shown in Figure 1(a), referred to as the \texttt{randQB} algorithm, which is most general and robust.  

%The {\randQB} algorithm can be further accelerated by utilizing a \emph{structured} random matrix, instead of the Gaussian matrix \cite{Halko2011}. This reduces the computational cost for a general dense $m\times n$ matrix from $O(mnl)$ to $O(mn\log(l))$ flops. However, this kind of technique usually % does not adapt to the fixed-precision problem and 
%cannot achieve $O(mn\log(l))$ complexity while combined with the power scheme, for handling matrices with slow decay of singular values. 

The truncated QR factorization with column pivoting can also be used for low-rank matrix approximation \cite{MatrixBook, Gu1996}. Classical pivoted QR factorization has the disadvantage that it is hard to be parallelized or to take usage of BLAS-3 operation. Recently, the pivoted QR factorization was largely accelerated through utilizing a randomized technique, which achieves the efficiency comparable to the unpivoted QR decomposition \cite{Martin2015b, Duersch2015}. This makes the truncated pivoted QR factorization competitive for low-rank approximation. Notice that QR factorization can be regarded as a special case of the QB factorization. Therefore, one of the techniques proposed here (c.f. Sect. 3) could benefit the solution of the fixed-precision problem based on QR factorizations, as well.

The major computation of the {\randQB} algorithm lies at the multiplication of $\bA$ and $\boldsymbol{\Omega}$. If $\bA$ is sparse or a structured matrix (often \emph{implicitly} defined) for which matrix-vector products can be rapidly evaluated, the cost of multiplication can be largely reduced (even to $O(m+n)$ flops). The implicitly-defined structured matrix often arises from physical problems, such as a discretized integral operator applied via the fast multipole method, and is sometimes referred to as an \emph{implicit sparse matrix}. 
%In this sparse setting, an interesting topic of research is the technique of using sparse random projection or sparsity-preserving random projection \cite{sparseJL}, which is beyond the scope of this paper. 

For the fixed-precision problem, an adaptive randomized range finder was proposed in \cite{Halko2011}. It employs the incremental sampling approach with a probabilistic error estimator to determine the size of $\boldsymbol{Q}$ and $\boldsymbol{B}$. 
It is based on the statement that
\begin{equation}\label{eq: stochastic_error_est}
\|\bA-\bQ\bQ^T\bA \|_2 \le 10 \sqrt{\frac{2}{\pi}} \max_{i=1, \cdots, r} \|(\bA-\bQ\bQ^T\bA)\boldsymbol{\omega}^{(i)} \| ,
\end{equation}
with probability at least $1-10^{-r}$ [2]. Here $\|\cdot\|_2$ stands for the spectral norm, $\boldsymbol{\omega}^{(i)}$ is a random vector, and $r$ is a small integer, e.g., $r=10$. However, the error estimator often overestimates the approximation error, yielding much larger output matrices than what is necessary.

The {\randQBb} algorithm in \cite{Martin2015} is based on the single-vector version of {\randQB} algorithm with Gram-Schmidt procedure, which allows to construct the QB factorization and to evaluate its error step by step.     
%It is shown in Figure 2(a), which constitutes the ``single-vector randQB'' algorithm for computing the QB approximation.
%If lines (4) and (5) are replaced such that we pick $\boldsymbol{q}_i$ as the largest column of $\bA^{(i)}$, the algorithm becomes the column pivoted Gram-Schmidt algorithm for partial QR factorization. 
%It is straight-forward to show that if the algorithm is executed in exact arithmetic, $\boldsymbol{Q}_i$ is orthonormal, $%\boldsymbol{B}_i=\boldsymbol{Q}_i^T \boldsymbol{A}$, and
%\begin{equation}\label{eq:Ai}
%\boldsymbol{A}^{(i)}= \boldsymbol{A}- \boldsymbol{Q}_i\boldsymbol{B}_i ~.
%\end{equation}
%Therefore, the error of QB approximation is updated, allowing a precise stopping criterion for the auto-rank problem. 
%If we choose $\boldsymbol{q}_i$ in line (4) through drawing a random vector and multiplying $\boldsymbol{A}$ with it, we obtain the "single-vector randQB" algorithm for computing the QB approximation. 
%The algorithm is mathematically equivalent to the basic randQB procedure shown in Figure 1, except for the stopping criterion used. 
In order to exploit \emph{blocking} to attain high performance of linear algebraic computation, the algorithm is then converted to the blocked randQB algorithm in Figure 1(b), where step (6) is for calculating the residual matrix. Therefore, the {\randQBb} algorithm allows a precise error calculation for the fixed-precision problem. Notice that step (4) there is a re-orthogonalization operation which eases the accumulation of round-off error under floating point arithmetic. 
%The experiments showed that the {\randQBb} has simular runtime and accuracy as the basic \randQB
% than the column-pivoted QR factorization, and runs much faster on multi-core architectures.

The blocked randQB algorithm is able to produce $\bQ$ and $\boldsymbol{B}$ in smaller sizes than the adaptive randomized range finder. However, explicitly maintaining the residual matrix brings extra time and memory cost if a large matrix $\bA$ is handled. This disadvantage becomes more serious if $\bA$ is also sparse, because the fill-in phenomena leads to a dense residual matrix.

In order to reveal the difference among the relevant randomized algorithms, we give a brief comparison of them for the fixed-rank problem, presented as Table 1. 
\texttt{svds} denotes the Matlab built-in command for truncated SVD \cite{Lehoucq1998}, which is based on a Krylov subspace iterative method. 
{\qrcp} denotes the randomized pivoted QR factorization \cite{Duersch2015}.  {\randQBfp} is one of the contributions in this paper, which is a pass-efficient algorithm (c.f. Sect. 4). 
To depict the performance of the algorithms in the situation where the cost of accessing matrix entries is expensive (e.g. $\bA$ are stored in slow memory) \cite{Halko2011}, we include the number of passes over matrix $\bA$ in Table 1.
From the table, we see that the  {\randQBfp} algorithm inherits the merits of the  {\randQB} algorithm, and is also suitable for the fixed-precision problem.  
%It also tends to be a kind of \emph{pass-efficient} algorithm, which requires only a constant number of passes over the data, as opposed to $O(k)$ passes for classical algorithms. This is very useful for the situations where the cost of accessing matrix entries is expensive. Particularly, a single-pass algorithm is desirable for data generated in a streaming fashion or data that are too large to be stored in fast memory \cite{Halko2011, ZhangZH2016}.
%
%For the fixed-rank problem of QB factorization, the basic {\randQB} procedure and its blocked variant {\randQBb} \cite{Martin2015} given in Figure 1(b) produce similar results. The pivoted QR decomposition also serves the purpose. Especially, the recent truncated QRCP algorithm using random sampling (denoted by \qrcp) largely reduces its runtime \cite{Duersch2015}. A contribution of this paper is a pass-efficient algorithm called {\randQBfp}. The properties of these algorithms for the fixed-rank problem are summarized in Table 1. Actually, the $1+2P$ of {\randQBfp} can be further reduced because it is possible to apply less frequent orthogonalization during the application of power scheme \cite{Voronin2015}.
\begin{table}[h]
\begin{threeparttable}
\caption{Properties of the randomized algorithms for the fixed-rank matrix approximation ($P$ is a small integer for power scheme, and $b$ is the block size).}
\centering
\begin{tabular}{cccccc}
\hline
	& \texttt{svds} & \randQB & \randQBb & \qrcp & \randQBfp \\
\hline
Computational efficiency & low & high & high & high & high \\
% \hline
% Fixed rank approx. accuracy & optimal & medium & medium & good & medium \\
\hline
Adaptive rank determination & no & no & yes & yes\tnote{$\dagger$} & yes \\
%\hline
%Adaptability to sparse matrix & yes & yes & no & yes & yes \\
\hline 
Number of passes over matrix & $\alpha k$\tnote{*} &  ~ 2 or 2+2$P$ & $4k/b$ & $3k/b+1$ & 1 or 1+2$P$\tnote{$\ddagger$} \\
\hline
\end{tabular}
\begin{tablenotes}
\item[*] \footnotesize{$\alpha$ is a number larger than 1.}
\item[$\dagger$] This was not mentioned in \cite{Duersch2015}. For more detail, please see Remark 3.2.
\item[$\ddagger$] Suppose matrix $\bA$ is stored in the row-major format.
\end{tablenotes}
\end{threeparttable}
\end{table}

    \section{An efficient Frobenius-norm error indicator and its application} In this section, we first propose an error indicator for measuring the approximation error in Frobenius norm and an efficient framework for solving the fixed-precision problem of QB factorization. Then, the {\randQBei} algorithm is derived based on the blocked randQB algorithm. Finally, the accuracy and validity of the error indicator in floating point arithmetic is discussed.
    \subsection{An error indicator} We first give a theorem regarding Frobenius norm of the error of the QB approximation.
    \begin{Mytheorem}\label{error_estimation}
        Let $\bA$ be an $m\times n$ matrix. $\bQ$ denotes an $m\times k$  orthonormal matrix ($k<m$), and $\bB = \bQ^\nT\bA$. Then, 
        \begin{equation}\label{eq:err_indicator}
            \|\bA-\bQ\bB\|^2_\nF = \|\bA\|^2_\nF - \|\bB\|^2_\nF.
        \end{equation}
    \end{Mytheorem}
    
    \begin{proof}
        Due to the property of Frobenius norm, we know for any matrix $\boldsymbol{M}$, 
\begin{equation}\label{eq:trace}
        \|\boldsymbol{M}\|^2_\nF = tr(\boldsymbol{M}^\nT \boldsymbol{M}),
\end{equation}
         where $tr(\cdot)$ calculates the trace of a matrix. Because $\bQ$ is orthonormal and $\bB = \bQ^\nT\bA$,
         \begin{equation}\label{eq:A_QB}
         \begin{aligned}
        (\bA - \bQ\bB)^\nT(\bA-\bQ\bB) &= (\bA-\bQ\bQ^\nT\bA)^\nT(\bA-\bQ\bQ^\nT\bA)\\
			&= \bA^\nT\bA - 2\bA^\nT\bQ\bQ^\nT\bA + \bA^\nT\bQ\bQ^\nT\bQ\bQ^\nT \bA \\
			&= \bA^\nT\bA  - \bA^\nT\bQ\bQ^\nT\bA\\
			&= \bA^\nT\bA - \bB^\nT\bB.
			\end{aligned}
			\end{equation}
        Now, applying the trace operation to both sides of (\ref{eq:A_QB}), and according to (\ref{eq:trace}), we obtain (\ref{eq:err_indicator}).
    \end{proof}
    
Theorem \ref{error_estimation} suggests that, if we have access to $\bB$ and $\|\bA\|_\nF$ is known \emph{a prior}, we can calculate the error of QB approximation without referring to the residual matrix. This leads to a framework for solving the fixed-precision problem, presented as Algorithm 1. It suits to any algorithm that incrementally generates the rows of $\bB$, including  {\randQBb}, {\qrcp} and the {\randQBfp} algorithm presented in Sect. 4.
        \begin{algorithm}
        \caption{A framework for solving the fixed-precision QB factorization problem}
        \label{algnew1}
        \begin{algorithmic}[1]
            \REQUIRE an $m\times n$ matrix $\bA$; desired accuracy tolerance $\varepsilon$.
            \ENSURE $\bQ$,~ $\bB$, ~such that $\|\bA- \bQ \bB\|_\nF< \varepsilon$.
            \STATE $\boldsymbol{Q} = [~]; ~ \boldsymbol{B} = [~];$  \COMMENT{empty matrices}
            \STATE $E = \|\boldsymbol{A}\|^2_\nF$  \COMMENT{initialization of the error indicator}
            \FOR{$i= 1, 2, 3, \cdots$}

            \STATE Generate $\bQ_i$ and $\bB_i$, s.t. $[\bQ, \bQ_i]$ is orthonormal and $\bB_i = \bQ_i^\nT\bA.$\label{refline}
            \STATE $\boldsymbol{Q} = [\boldsymbol{Q}, ~\boldsymbol{Q}_i]$
            \STATE $\boldsymbol{B} = \left[ \begin{array}{c}
                 \boldsymbol{B}\\
                 \boldsymbol{B}_i
            \end{array}\right]$
            \STATE $E = E - \|\boldsymbol{B}_i\|^2_\nF$    \COMMENT{update the error indicator}
            \STATE \textbf{if} $E < \varepsilon^2$ \textbf{then stop}
            \ENDFOR
            %\STATE $k= it$
        \end{algorithmic}
        \end{algorithm}
Below we prove the  correctness of Algorithm 1. %In all proofs thereafter regarding an iterative algorithm, we use $v^{(i)}$ to denote the value of any variable $v$ \emph{after} the $i$-th iteration of the loop is executed.
     
    % Now, we give the following proposition.
        \begin{Mytheorem} \label{pro:alg1}
            After the $i$-th iteration of loop in Algorithm 1 is executed, %$E^{(i)}$ satisfies
            \begin{equation}\label{thm2_eq_block}
                 E^{(i)}= \|\bA - \bQ^{(i)}\boldsymbol{B}^{(i)}\|^2_\nF,
            \end{equation}
            where $E^{(i)}$, $\bQ^{(i)}$ and $\bB^{(i)}$ denote the values of $E$, $\bQ$ and $\bB$ after the $i$-th iteration of the loop is executed, respectively.
        \end{Mytheorem}
        \begin{proof} Based on step 7 of Algorithm  1 and the property of Frobenius norm,
\begin{equation}\label{eq:Eformula}
                E^{(i)} = \|\bA\|^2_\nF - \sum_{j=1}^{i} \|\bB_j\|^2_\nF = \|\bA\|^2_\nF - \|\bB^{(i)}\|^2_\nF.
\end{equation}
Because $\bQ^{(i)}$ is an orthonormal matrix and $\bB^{(i)}=\bQ^{(i)\nT}\bA$,
(\ref{thm2_eq_block}) can be obtained by applying Theorem 1.
        \end{proof}

From Theorem \ref{pro:alg1}, we see that $E$ in Algorithm 1 equals to the square of the approximation error. It is an error indicator updated through calculating the Frobenius norm of $\boldsymbol{B}_i$. This yields two benefits: we no longer need to maintain the residual $\bA - \bQ\bB$, and the approximation error can be calculated with very small cost.
%computational cost of making error estimation becomes negligible.

Steps 7 and 8 in Algorithm 1 can be replaced by a row-by-row calculation scheme. 

\begin{tabular}{l}
\toprule
7a: \textbf{for} $j= 1, 2, \cdots m_i$ \textbf{do} ~ ~ ~ \# $m_i$ denotes the number of rows of $\bB_i$ \\
7b: \quad $E = E - \|\boldsymbol{B}_i(j, :)\|^2_\nF$ \\
8a: \quad \textbf{if} $E < \varepsilon^2$ \textbf{then} \\
8b: \quad \quad remove $\bB_i(j+1:m_i, :)$ from $\bB$; remove $\bQ_i(:, j+1:m_i)$ from $\bQ$  \\
8c: \quad \quad \textbf{stop} \\
8d: \quad \textbf{end if} \\
8e: \textbf{end for} \\
\bottomrule
\end{tabular}

\noindent{This adds negligible cost, but allows us to determine the certain row of $\boldsymbol{B}_i$ where the accuracy tolerance is just attained. It makes the column (row) number of outputted $\bQ$ ($\bB$) an arbitrary integer, instead of a multiple of block size $b$ in {\randQBb} algorithm.}

\textbf{\textit{Remark} 3.1}. If $\bA$ is an implicit sparse matrix, the proposed framework needs more effort for calculating its Frobenius norm. The columns of $\bA$ can be solved by multiplying $\bA$ with the canonical basis vectors. Therefore, the cost for calculating $\| \bA \|_\nF$ will be $O(n(m+n))$ flops, if each matrix-vector product costs $O(m+n)$ flops. This might be affordable, as it is executed just once.

\subsection{The {\randQBei} algorithm}
The combination of Algorithm 1 and the {\randQBb} algorithm results in Algorithm 2 (called {\randQBei}), whose steps 4$\sim$7 replace step 4 in Algorithm 1. Notice that step 7 in Algorithm 2 looks the same as step (5) in {\randQBb} algorithm, but is actually different. And, step (3) in  {\randQBb} algorithm becomes step 5 in Algorithm 2, which is the blocked Gram-Schmidt orthogonalization of $\bA \boldsymbol{\Omega}_i$ (notice $\bB= \bQ^T \bA$).
 %This treatment inherits the original spirit of the randQB procedure.
        \begin{algorithm}
        \caption{The {\randQBei} algorithm for the fixed-precision problem }
        \label{alg:randQBei}
        \begin{algorithmic}[1]
            \REQUIRE an $m\times n$ matrix $\bA$; desired accuracy tolerance $\varepsilon$; block size $b$.
            \ENSURE $\bQ$,~ $\bB$, such that $\|\boldsymbol{A}- \boldsymbol{QB}\|_\nF< \varepsilon$.
            \STATE $\boldsymbol{Q} = [~]; ~ \boldsymbol{B} = [~];$ 
            \STATE $E = \|\boldsymbol{A}\|^2_\nF$
            \FOR{$i= 1, 2, 3, \cdots$}
            \STATE $\boldsymbol{\Omega}_i$ = randn($n, b$) 
            \STATE $\boldsymbol{Q}_i$ = orth($\bA\boldsymbol{\Omega}_i-\bQ(\boldsymbol{B}\boldsymbol{\Omega}_i$))  
            \STATE $\bQ_i$ = orth($\bQ_i- \bQ(\bQ^\nT\bQ_i)$)  \quad \quad \quad  \# re-orthogonalization
            \STATE $\boldsymbol{B}_i=\bQ_i^\nT\bA$  \quad \quad \quad \quad \quad \quad \quad \quad \quad \quad \# no need to calculate $\bQ_i^\nT(\bA-\bQ\bB)$
            \STATE $\boldsymbol{Q} = [\boldsymbol{Q}, ~\boldsymbol{Q}_i]$
            \STATE $\boldsymbol{B} = \left[ \begin{array}{c}
                 \boldsymbol{B}\\
                 \boldsymbol{B}_i
            \end{array}\right]$
            \STATE $E = E - \|\boldsymbol{B}_i\|^2_\nF$  
            \STATE \textbf{if} $E < \varepsilon^2$ \textbf{then stop}
            \ENDFOR
        
        \end{algorithmic}
        \end{algorithm}

Due to Theorem 2 and the orthogonality of $\bQ$, we have the following proposition. 
        \begin{MyProp}\label{thm_ei}
      The {\randQBei} algorithm (Algorithm 2) is equivalent to the {\randQBb} algorithm, when executed in exact arithmetic.
        \end{MyProp}

Assuming that multiplying two dense matrices of sizes $m\times n$ and $n\times l$ costs $C_{mm}mnl$ flops, and performing an economic QR factorization of an $m\times n$ dense matrix costs $C_{qr}mn\min(m,n)$ flops, we can analyze the flop counts of the relevant algorithms and compare their performance for handling a \emph{dense} $\bA$. If $T_{randQB}$ and $T_{randQB\_b}$ denote the runtime of the algorithms \texttt{randQB}  and \texttt{randQB\_b} respectively~\cite{Martin2015},
\begin{equation}\label{time_randqb}
T_{randQB} \sim 2C_{mm}mnl+ C_{qr}ml^2~ ,
\end{equation}
\begin{equation}\label{time_randqb_b}
T_{randQB\_b} \sim 3C_{mm}mnl+C_{mm}ml^2+\frac{2}{t}C_{qr}ml^2 ,
\end{equation}
where $l$ is the number of columns in the resulting matrix $\bQ$, and $t$ satisfies $l=tb$. Note that $b$ is practically much smaller than $l$ (say, $b=$10 or 20), although the optimal choice of block size depends strongly on what hardware is
used \cite{Martin2015}.
 %Note that $l$ can be a little bit larger than $k$ if a rank-$k$ truncated SVD is pursued.

%For the single-vector randQB algorithm in Figure 2(a), similar analysis applies. It is easy to see that its runtime
%\begin{equation}\label{time_randqb_sv}
%T_{randQB\_sv} \sim (3C_{mm}+1)mnl~ .
%\end{equation}
%Here, ``+1'' stands for the matrix subtraction in line (9) of the single-vector randQB algorithm. 

For the {\randQBei} algorithm, the runtime can be similarly depicted:
\begin{equation}\label{time_alg1}
%\begin{split}
                \begin{aligned}
T_{randQB\_EI} %\sim 2C_{mm}mnl + C_{mm}(4m+2n)b^2\sum_{i=1}^{t-1}i+ \frac{2}{t}C_{qr}ml^2 \\
                 \sim 2C_{mm}mnl+ C_{mm}(2m+n)l^2+\frac{2}{t}C_{qr}ml^2~.
%\end{split}
                \end{aligned}
\end{equation}
Because $l$ is usually much smaller than $m$ and $n$, we see that the flop count of {\randQBei} is about 2/3 of that of {\randQBb}, and is comparable to that of the basic randQB algorithm. Notice that $C_{qr}$ is several times larger than $C_{mm}$. So, the flop count of {\randQBei} could be smaller than that of {\randQB} in the situation where $t$ is a large number. If $\bA$ is sparse, this would more possibly happen, because the \texttt{randQB} algorithm loses the benefit brought by the BLAS-3 operation.

% However, the flop count is not the only factor affecting the actual performance of an algorithm. Algorithm 1 updates a vector per iteration, which means the \emph{blocking} is not used for the loop. This is not optimal with respect to the runtime efficiency. And, this Gram-Schmidt procedure based scheme is susceptible to numerical round-off error. The both issues are resolved by the blocked version of the algorithm. We will discuss them in Section 4.

The advantage of {\randQBei} over  {\randQBb} becomes more prominent if $\bA$ is a sparse matrix. With the proposed error indicator, we no longer need the residual matrix. In contrast, it is always a dense matrix in the {\randQBb} algorithm, and costs much larger memory and induces much more computations. % than the sparse $\bA$. %Besides, since the error matrix is not stored, the memory saving of {\randQBei} algorithm becomes significant if $\bA$ is of large size.

\textbf{\textit{Remark} 3.2}. Algorithm 1 can also be combined with the {\qrcp} algorithm \cite{Duersch2015}. Although {\qrcp} generates a column permutation matrix as well, it does not affect the Frobenius norm of each partial or the entire matrix of $\boldsymbol{B}$. 
Therefore, this will produce another efficient algorithm for adaptive low-rank matrix approximation, which also adapts to sparse matrices.

\subsection{Floating point arithmetic}
Below we discuss the accuracy of the error indicator in floating point arithmetic. We use $\mathcal{E}(\cdot)$ and $\mathcal{E}_r(\cdot)$ to denote the functions of error and relative error, respectively.
%$\hat{x}$ to denote the actual (calculated) value of any theoretical quantity $x$. And, 

As the error indicator $E=\|\bA\|^2_\nF- \|\boldsymbol{B}\|^2_\nF$, its calculated value $\hat{E}$ cannot be accurate when $E$ is very small, due to the cancellation in calculation.
 In floating-point arithmetic, the machine precision $\epsilon_{mach}$ characterizes the maximum relative error of converting a real number to its floating-point representation, i.e.
\begin{equation} \label{eq:epsilon_mach}
\forall x, ~ ~ |\mathcal{E}_r(x)| \le \epsilon_{mach}.
\end{equation}
According to the definition of Frobenius norm, $\| \bA \|^2_\nF$ is the summation of squares of matrix entries. Therefore, the relative error of $\| \bA \|^2_\nF$ is bounded by $2\epsilon_{mach}$. The same thing applies to $\| \bB \|^2_\nF$. So,
\begin{equation}
\begin{aligned}
|\mathcal{E}(E)|=|\mathcal{E}(\|\bA\|^2_\nF- \|\boldsymbol{B}\|^2_\nF)| &\le |\mathcal{E}(\|\bA\|^2_\nF)|+|\mathcal{E}(\|\boldsymbol{B}\|^2_\nF) |\\
 &\le  2\epsilon_{mach}(\|\bA\|^2_\nF+ \|\bB\|^2_\nF) \\
 &<  4\epsilon_{mach}\|\bA\|^2_\nF .
\end{aligned}
\end{equation}   
%where the last inequality holds except $E$ is extremely small. 
If we want to guarantee that $E$ has a relative error no more than $\delta$, i.e., $|\mathcal{E}(E)| \le \delta E$, we shall enforce
\begin{equation}
\begin{aligned}
4\epsilon_{mach}\|\bA\|^2_\nF \le \delta E.
\end{aligned}
\end{equation}   
This means the preset accuracy tolerance $\varepsilon$, which is larger than $\sqrt{E}$  at the termination of the algorithm, should satisfy:
\begin{equation}
\begin{aligned}
\varepsilon > \sqrt{E} \ge \sqrt{\frac{4\epsilon_{mach}\|\bA\|^2_\nF}{\delta}}=\sqrt{\frac{4\epsilon_{mach}}{\delta}}\|\bA\|_\nF.
\end{aligned}
\end{equation} 

%With $\epsilon_{mach} \approx 1.11 \times 10^{-16}$ in double precision, 
So, we obtain the following Theorem.
    \begin{Mytheorem}\label{error_indicator}
        Suppose matrix $\bA$ and accuracy tolerance $\varepsilon$ are the input to the {\randQBei} algorithm. If $\varepsilon > \sqrt{\frac{4\epsilon_{mach}}{\delta}}\|\bA\|_\nF$, the relative error of the calculated error indicator $E$ must be no more than $\delta$.  E.g., if $\varepsilon > 2.1\times 10^{-7}\|\bA\|_\nF$, the error of $E$ is within $1\%$ in the double-precision floating arithmetic, where $\epsilon_{mach} \approx 1.11 \times 10^{-16}$.
    \end{Mytheorem}

Notice that, as an error indicator for the fixed-precision problem, $E$ should have sufficient accuracy (e.g., with relative error $\sim 1\%$ or less). Otherwise, the outputted QB factorization would not satisfy the preset accuracy tolerance.

Besides, the orthogonality of $\bQ$ also affects the accuracy of $E$. As the number of columns in $\bQ$ increases, its orthogonality gradually degrades. This issue occurs for $\bQ$ produced either by a single run of QR factorization (based on Householder transformation) or by a Gram-Schmidt procedure followed by the re-orthogonalization step. We will investigate its effect in the following experiment.

%Lastly, $\|\bA\|^2_\nF$ and $\|\bB\|^2_\nF$ (equivalent to the summation in (\ref{eq:Eformula})) should be carefully calculated. Otherwise, their relative error may exceed $2\epsilon_{mach}$, inducing extra error. For example, we shall not calculate the Frobenius norm and then its square. Instead, directly calculating the sum of squares of matrix entries deliver more accurate result.  

%The following numerical experiment compares the values of error indicator and the actual error of QB approximation. 
An $n\times n$ matrix $\bA$ is constructed to have singular values according to a decaying exponential. Two instances are tested, with singular value $\sigma_j= e^{-j/20}$ and $\sigma_j= e^{-j/200}$, $j= 1, 2, \cdots $, respectively. The results obtained from executing {\randQBei} algorithm ($b=10$) and {\randQB} algorithm with different values of rank parameter $l$ are shown in Figure 2. Note that for some large value of $l$, the error indicator can be of negative value, such that it cannot be drawn in the log-scale plot. From the figure, we can validate the correctness of Theorem 3. Providing that the square of error $\|\bA- \bQ\bB\|^2_\nF / \|\bA\|^2_\nF > (2.1\times 10^{-7})^2 = 4.4\times 10^{-14}$, the error indicator matches the square of error very well. This holds even when $l$ is larger than $3,000$, which corresponds to the situation with larger accumulated round-off error. In Figure 2, \begin{figure*}[h]
\centering
\subfigure[$n=3,000,~ \sigma_j= e^{-j/20}$] {\includegraphics[width=2.52in]{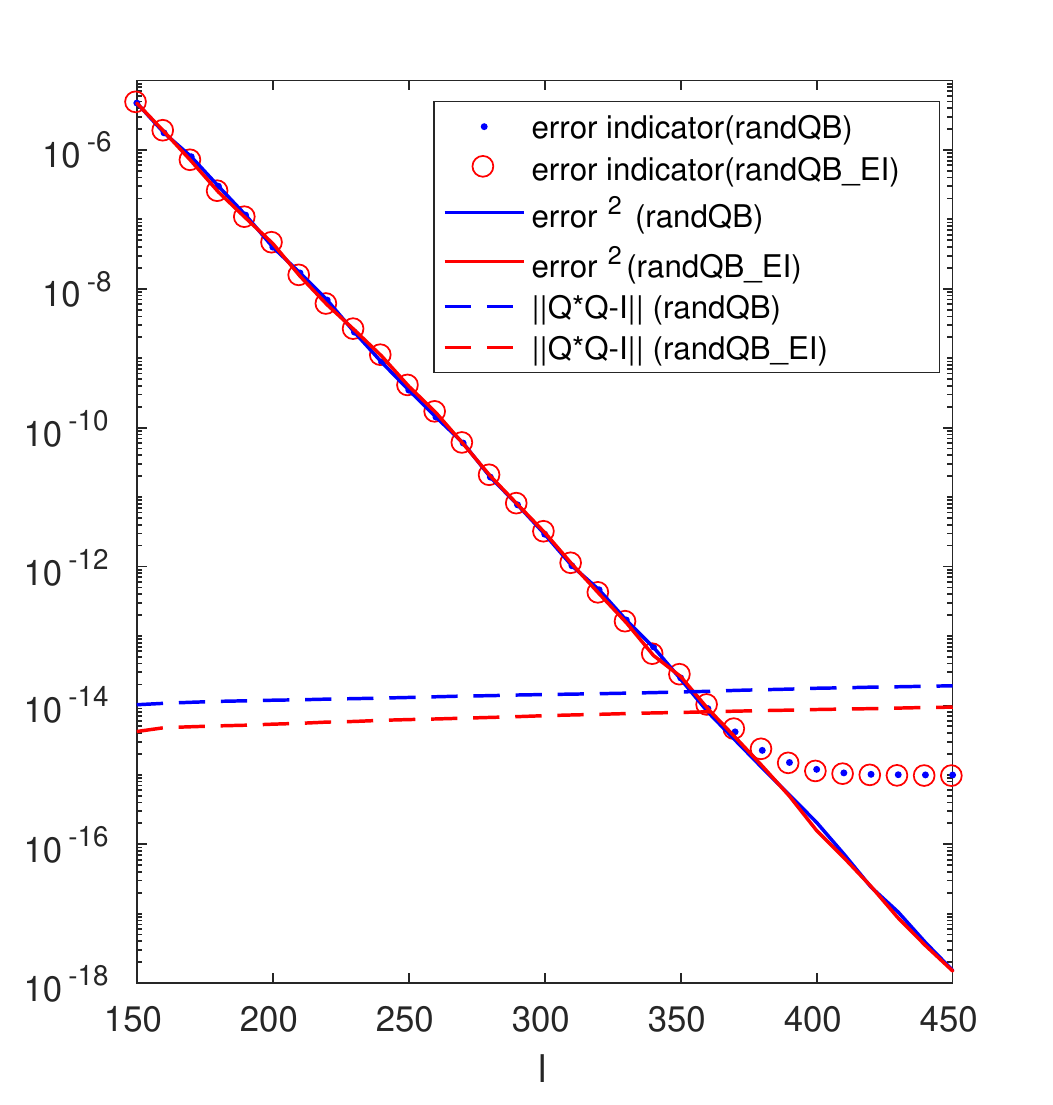}}
\subfigure[$n=20,000,~ \sigma_j= e^{-j/200}$] {\includegraphics[width=2.52in]{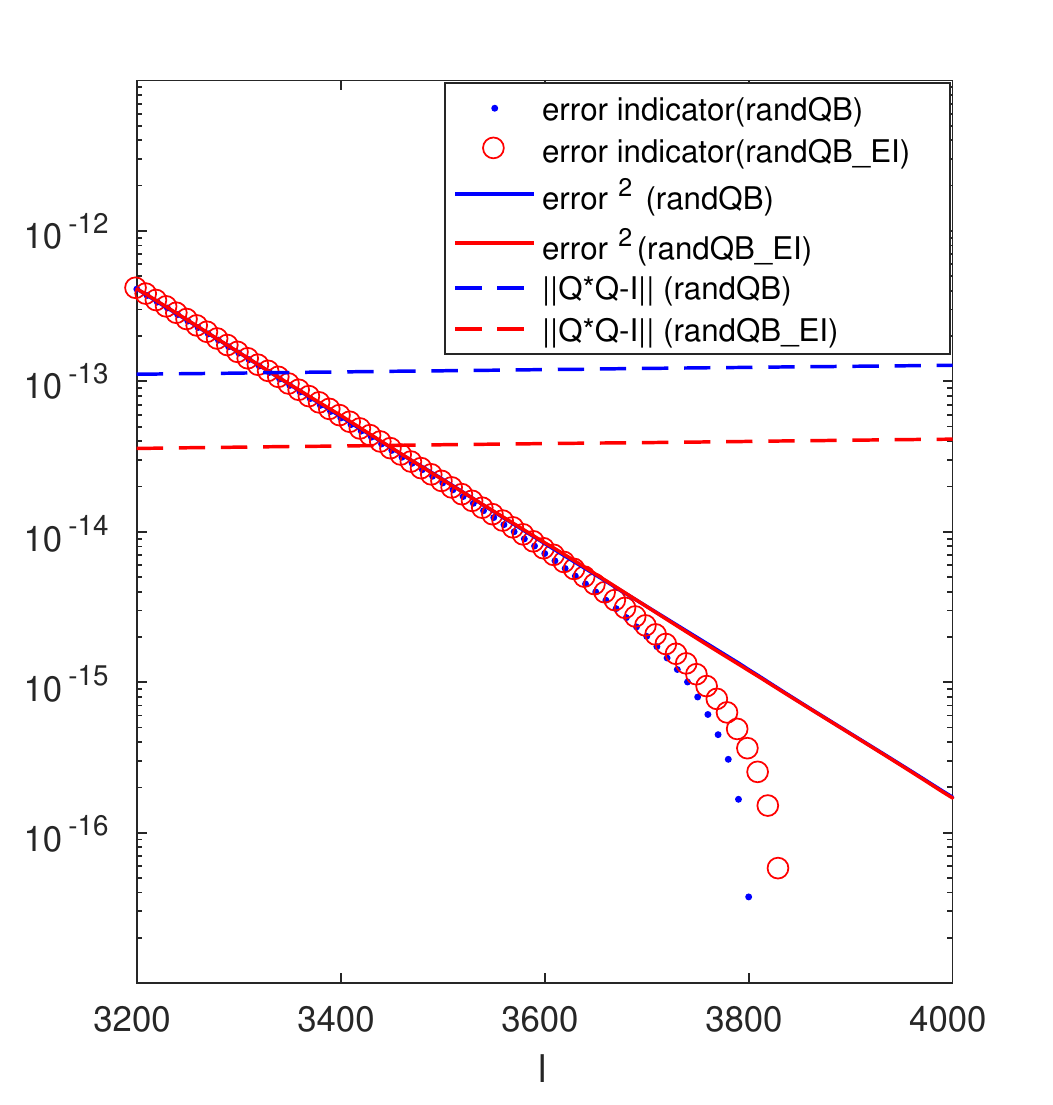}}
\caption{Normalized error indicator $E / \|\bA\|^2_\nF$, error square $\|\bA- \bQ\bB\|^2_\nF / \|\bA\|^2_\nF$ and the loss of orthogonality of $\bQ$ v.s. $\bQ$'s column number $l$ in {\randQB} algorithm and {\randQBei} algorithm.}
\label {fig2}
\end{figure*} the value of $\|\bQ^\nT \bQ- \bI \|_\infty$ \footnote{This measure of loss of orthogonality follows Cleve Moler's blog with title ``Compare Gram-Schmidt and Householder Orthogonalization Algorithms'' posted on Oct. 17, 2016.} 
 is also plotted, which reveals the loss of orthogonality of $\bQ$. The results show that this issue is not severe in both the {\randQB} and {\randQBei} algorithms, although it gradually increases as the columns of $\bQ$ are increased.
%\begin{figure*}[h]
%\centering
%\subfigure[$n=3000, \sigma_j= e^{-j/20}$] {\includegraphics[width=2.52in]{figEIa3.eps}}
%\subfigure[$n=4000, \sigma_j= e^{-j/80}$] {\includegraphics[width=2.52in]{figEIb3.eps}}
%\caption{Normalized error indicator $E / \|\bA\|^2_\nF$, error $\|\bA- \bQ\bB\|^2_\nF / \|\bA\|^2_\nF$ and the loss of orthogonality of $\bQ$ v.s. the column number $l$ of matrix $\bQ$ in {\randQB} algorithm and {\randQBei} algorithm.}
%\label {fig3}
%\end{figure*}

\textbf{\textit{Remark} 3.3}. Theorem 3 suggests the limitation of the error indicator and the proposed algorithms for the fixed-precision problem. It means that the efficient framework for adaptive rank determination would not work, in double-precision floating arithmetic, for the problem with the accuracy tolerance $\varepsilon$ less than $2.1\times 10^{-7}\|\bA\|_{\nF}$.

\section{A pass-efficient algorithm for the QB factorization}

    The technique in last section efficiently solves the fixed-precision problem measured in Frobenius norm. However, the {\randQBei} algorithm is not suitable for the scenarios where accessing matrix entries is very expensive (e.g., $\bA$ is too large and has to be stored on hard disk), as it visits $\bA$ for considerable times (see Table 1).
%handling very large matrix stored in slow memory, as it repeatedly visits the input matrix. 
In this section, we propose a pass-efficient algorithm which largely reduce the number of passes over $\bA$.   
    
%    As we have mentioned, most randomized blocked algorithms visit matrix $\bA$ for $O(k/b)$ times. They lose the pass-efficient property, which is favorable for handling large matrices.     
%    In this section, we derive a pass-efficient blocked algorithm for QB factorization (called {\randQBfp}). It is equivalent to the {\randQBb} algorithm, but reduces the number of passes over $\bA$ to the fewest. We will also discuss its usage in solving the fixed-precision problem.

    \subsection{The version without re-orthogonalization}
We first consider the fixed-rank problem where the rank parameter $k$ is given. A preliminary pass-efficient algorithm (presented as Algorithm 3) can be derived from the {\randQBb} or {\randQBei} algorithm. 
The steps correspond to those of {\randQBei} in Algorithm 2, one by one, except that the applications of $\bA$ are moved out of the loop, and step 6 for re-orthogonalization is ignored.
 In Algorithm 3,
 %the products of $\bA$ and random vectors, as well as $\bA^T\bA$ and the random vectors, are pre-computed in steps 3 and 4.  Then, 
 step 6 is the same as step 4 of {\randQBei} algorithm, and steps 7 and 8 correspond to step 5 of {\randQBei}. In step 9, $\bR_i^{-T}$ is the inverse of the transpose of an upper triangular matrix $\bR_i$. This step can be regarded as solving linear equations with the coefficient matrix $\bR_i^T$, which is implemented by ``$\setminus$'' operator in Matlab.
 % Its key point is to pre-compute the products of $\bA$ and random vectors, as well as the products of $\bA^T\bA$ and the random vectors. Then, in the loop we take columns from the pre-computed matrices $\boldsymbol{\Omega}, \boldsymbol{G}$ and $\boldsymbol{H}$, to get the needed information. For example, with $\boldsymbol{G}$ the 7th and 8th steps in Algorithm 2 perform the same function as step (3) in the \texttt{randQB\_b} algorithm. The equivalence between the 9th step in Algorithm 2 and step (5) in the \texttt{randQB\_b} algorithm is also straightforward. With the matrix-vector multiplications lumped into two matrix-matrix multiplications, more efficiency benefit from the BLAS-3 operation is expected while handling a dense $\bA$. Algorithm 2 has the property of \emph{pass-efficiency}, as the original {\randQB} algorithm \cite{Halko2011}. 
        \begin{algorithm}
        \caption{A preliminary pass-efficient algorithm for fixed-rank QB factorization}
        \label{alg3}
        \begin{algorithmic}[1]
            \REQUIRE $\boldsymbol{A}\in \mathbb{R}^{m\times n}$,  $k$, $s$, block size $b$.
            \ENSURE $\boldsymbol{Q}$, $\boldsymbol{B}$. %\final{of size $m\times k$ and $k\times n$}, ~s.t. \st{$\|\boldsymbol{A}- \boldsymbol{QB}\|$}\st{$<$}\st{$\varepsilon$.}\final{$\bA\approx \bQ\bB$.}
            \STATE $\boldsymbol{Q} = [~]; ~ \boldsymbol{B} = [~]; ~ l= k+s; $
%            \STATE $\boldsymbol{\Omega}= $ randn($n, \tilde{l}$), where $\tilde{l}$ is a sufficiently large number.
            \STATE $\boldsymbol{\Omega}= $ randn($n$, $l$)
        \STATE $\boldsymbol{G} = \boldsymbol{A}\boldsymbol{\Omega}$
        \STATE $\boldsymbol{H} = \boldsymbol{A}^T\boldsymbol{G}$
%            \STATE \st{$E = \|\boldsymbol{A}\|^2$}
            \FOR{$i=1, 2, 3, \cdots, l/b$}
            %\FOR{$\sqrt{E} > \varepsilon$}
            \STATE $\boldsymbol{\Omega}_i = \boldsymbol{\Omega}(:,~ (i-1)b+1:ib)$
            \STATE $\boldsymbol{Y}_i = \boldsymbol{G}(:,~ (i-1)b+1:ib) - \bQ(\boldsymbol{B\Omega}_i)$
			\STATE $[\boldsymbol{Q}_i,~\boldsymbol{R}_i]= $ qr($\boldsymbol{Y}_i$)
			\STATE $\boldsymbol{B}_i= \boldsymbol{R}_i^{-T}(\boldsymbol{H}(:,~ (i-1)b+1:ib)^T- \boldsymbol{\Omega}_i^T\boldsymbol{B}^T\boldsymbol{B}) $
            \STATE $\boldsymbol{Q} = [\boldsymbol{Q}, ~\boldsymbol{Q}_i]$
            \STATE $\boldsymbol{B} = \left[ \begin{array}{c}
                 \boldsymbol{B}\\
                 \boldsymbol{B}_i
            \end{array}\right]$
%                \STATE \st{$E = E - \|\boldsymbol{B}_i\|^2$}
%				\STATE \st{\textbf{if} $\sqrt{E} < \varepsilon$ \textbf{then stop}}
            \ENDFOR
        
        \end{algorithmic}
        \end{algorithm}
        
      % With the following proposition, we establish the correctness of Algorithm 2. 

Because $\bY_i= \bQ_i \bR_i$, 
\begin{equation}\label{eq:Qi_Yi_relation}
\bQ_i= \bY_i\bR_i^{-1} . 
\end{equation}
Substituting it into step 7 of the {\randQBei} algorithm, we have:
\begin{equation}
\begin{aligned}
\bB_i &= (\bY_i\bR_i^{-1})^T \bA \\
&= \bR_i^{-T} \bY_i^T \bA \\
&= \bR_i^{-T}(\bG(:,~ (i-1)b+1:ib)^T- \boldsymbol{\Omega}_i^T\boldsymbol{B}^T\bQ^T )\bA \\
&= \bR_i^{-T}((\bA^T \bG(:,~ (i-1)b+1:ib))^T- \boldsymbol{\Omega}_i^T\boldsymbol{B}^T\bQ^T \bA) \\
&= \bR_i^{-T} (\bH(:,~ (i-1)b+1:ib)^T- \boldsymbol{\Omega}_i^T\boldsymbol{B}^T\bB) .
\end{aligned}
\end{equation}
This means that step 9 of Algorithm 3 is equivalent to step 7 of {\randQBei}. Therefore, we obtain the following proposition. 
        
        \begin{MyProp}\label{thm1_block}
           Algorithm 3 is mathematically equivalent to the fixed-rank version of {\randQBei} algorithm without re-orthogonalization.
        \end{MyProp}

	   \subsection{The version with re-orthogonalization}  In reality, the loss of orthogonality among the columns of $\{\bQ_1,\bQ_2, \cdots \}$ occurs due to the accumulation of round-off error. This means the re-orthogonalization step, i.e. step 6 in Algorithm 2, cannot be ignored. Below we derive the revisions to Algorithm 3 to address this issue. 

The re-orthogonalization step can be expressed as
\begin{equation}\label{eq:re-orth1}
\scalebox{1}{$\tilde{\bQ_i}\tilde{\boldsymbol{R}_i}=\boldsymbol{Q}_i-\bQ\bQ^T\boldsymbol{Q}_i$},
\end{equation}
where $\tilde{\bQ_i} \neq \bQ_i$ and $\tilde{\boldsymbol{R}_i} \neq \boldsymbol{I}$ due to the loss of orthogonality. $\tilde{\bQ_i}$ is better orthogonal to the previously generated $\{\bQ_1, \bQ_2, \cdots, \bQ_{i-1}\}$ than $\bQ_i$.
Now, we need to derive a formula for calculating $\boldsymbol{B}_i$ which does not involve $\bA$ explicitly. Based on (\ref{eq:Qi_Yi_relation}), 
\begin{equation}\label{eq:tildeQi}
\scalebox{1}{$\tilde{\bQ_i}= (\boldsymbol{I}-\bQ\bQ^T) \boldsymbol{Y}_i \boldsymbol{R}_i^{-1} \tilde{\boldsymbol{R}}_i^{-1}$}.
\end{equation}
%Then, the new $\boldsymbol{B}_i$ is calculated as: 
% \! is used to adjust space before/after =, +, -.
So, the corresponding formula for $\bB_i$ is:
\begin{equation}\label{eq:newBi}
            \begin{aligned}
\scalebox{1}{$\tilde{\boldsymbol{B}_i}$} =& \scalebox{0.9}{$\tilde{\bQ}_i^T \bA$}  \\
=& \scalebox{1}{$(\tilde{\boldsymbol{R}_i}\boldsymbol{R}_i)^{-T} \boldsymbol{Y}_i^T(\boldsymbol{I}-\bQ\bQ^T) \bA$}  \\
=& \scalebox{1}{$(\tilde{\boldsymbol{R}_i}\boldsymbol{R}_i)^{-T} (\boldsymbol{\Omega}_i^T \bA^T - \boldsymbol{\Omega}_i^T \boldsymbol{B}^T \bQ^T) (\bA- \bQ\bB)$}, 
            \end{aligned}
\end{equation}
where the formula of $\boldsymbol{Y}_i$ and the equality $\bB=\bQ^T\bA$ is taken usage of.
The product of the last two brackets can be further simplified.
\begin{equation}\label{eq:newBi2}
            \begin{aligned}
& \scalebox{1}{$(\boldsymbol{\Omega}_i^T \bA^T - \boldsymbol{\Omega}_i^T \boldsymbol{B}^T \bQ^T) (\bA- \bQ\bB)$} \\
& = \scalebox{1}{$\boldsymbol{H}_i^T - \boldsymbol{G}_i^T \bQ\bB -\boldsymbol{\Omega}_i^T \bB^T\bB+ \boldsymbol{\Omega}_i^T \bB^T\bQ^T\bQ\bB $} \\
& = \scalebox{1}{$\boldsymbol{H}_i^T - \boldsymbol{Y}_i^T \bQ\bB - \boldsymbol{\Omega}_i^T \bB^T\bB $}.
            \end{aligned}
\end{equation}
In the deduction, $\boldsymbol{G}_i$ and $\boldsymbol{H}_i$ denote $\boldsymbol{G}(:,(i-1)b+1:ib)$ and $\boldsymbol{H}(:,(i-1)b+1:ib)$ in Algorithm 3, respectively. Therefore,
\begin{equation}\label{eq:newBi3}
            \begin{aligned}
\tilde{\boldsymbol{B}_i} = (\tilde{\boldsymbol{R}_i}\boldsymbol{R}_i)^{-T}(\boldsymbol{H}_i^T - \boldsymbol{Y}_i^T \bQ\bB - \boldsymbol{\Omega}_i^T \bB^T\bB)  ~.
            \end{aligned}
\end{equation} 

%take usage of the formula for $\boldsymbol{Y}_i$ shown as step 7 in Algorithm 2, and the equality $\boldsymbol{B}^{(i-1)}=\left(\bQ^{(i-1)}\right)^T\bA$ which may not hold for large value of $i$ due to round-off error. 

Based on (\ref{eq:re-orth1}) and (\ref{eq:newBi3}), we can derive the version with re-orthogonalization for Algorithm 3. We just need to replace the 9th step with the following steps.

\begin{tabular}{l}
\toprule
9a: $[\boldsymbol{Q}_i,~\tilde{\boldsymbol{R}_i}]= $ qr($\boldsymbol{Q}_i-\bQ(\bQ^T\bQ_i)$)\\ 
9b: $\boldsymbol{R}_i=\tilde{\boldsymbol{R}_i}\boldsymbol{R}_i$ \\
9c: $\boldsymbol{B}_i= \boldsymbol{R}_i^{-T}(\boldsymbol{H}(:,~ (i-1)b+1:ib)^T- \boldsymbol{Y}_i^T\bQ\boldsymbol{B} - \boldsymbol{\Omega}_i^T\boldsymbol{B}^T\boldsymbol{B}) $ \\
\bottomrule
\end{tabular}

\noindent{Notice that $\bQ_i$ and $\boldsymbol{B}_i$ are overwritten to stand for $\tilde{\bQ_i}$ and $\tilde{\boldsymbol{B}_i}$, respectively.}
Based on Proposition 2 and the above deduction, we see that the pass-efficient algorithm with re-orthogonalization is also mathematically equivalent to the fixed-rank versions of \texttt{randQB\_EI} and \texttt{randQB\_b} algorithms.
%        \begin{proposition}\label{thm2_block}
%      Algorithm 4 is algebraically equivalent to the fixed-rank versions of {\randQBei} algorithm and {\randQBb} algorithm, when executed in exact arithmetic.
%        \end{proposition}

This algorithm with \emph{fewer passes} over $\bA$ is called \texttt{randQB\_FP}. 
%Actually, it is a single-pass algorithm, because 
%steps 3 and 4 can be implemented with only one pass over $\bA$. %In contrast, the existing blocked algorithms need $O(l/b)$ passes over $\bA$. 
% The re-orthogonalization induces some extra computation. 
Based on the notations in Section 3, its flop count analysis is as follows.
\begin{equation}\label{time_alg4}
%\begin{split}
                \begin{aligned}
T_{randQB\_FP} % \sim 2C_{mm}mnl + 4C_{mm}(m+n)b^2\sum_{i=1}^{t-1}i+ 2C_{qr}mb^2t \\
                 \sim 2C_{mm}mnl+ 2C_{mm}(m+n)l^2+ \frac{2}{t}C_{qr}ml^2 ~,
%\end{split}
                \end{aligned}
\end{equation}
where $t$ satisfies $l=tb$. % For simplicity, we do not differentiate $l$ and $\tilde{l}$, although the latter is often larger than the former.
% The strategy regarding the $\tilde{l}$ in actual implementation will be discussed in Section 6.
Compared with the {\randQBei} algorithm, the \texttt{randQB\_FP} algorithm has slightly larger flop count. However, while handling a dense $\bA$ its actual runtime may be shorter because it lumps the multiplications with $\bA$. % and the benefit of BLAS-3 operations.
%And, the \texttt{randQB\_FP} algorithm owns the unique single-pass property which is favorable for an extremely large $\bA$.
 
\textbf{\textit{Remark} 4.1}. The round-off error may affect the accuracy of $\bB_i$, and it increases as the number of iterations increases. However, this may not be an issue for practical low-rank approximation problems.  
In Section 5, we will present numerical experiments to validate the effectiveness of the {\randQBfp} algorithm, which shows it works very well for many applications with the rank parameter up to several thousands or the relative Frobenius-norm error of approximation as small as $10^{-7}$.

\textbf{\textit{Remark} 4.2}. The {\randQBfp} algorithm can derive a single-pass algorithm, if matrix $\bA$ is stored in the row-major format or is revealed row(s) by row(s).
Suppose $\bA_{i,:}$ denotes the $i$-th row of $\bA$. With it we have the  $i$-th row of $\bG$, $\bG_{i,:}=\bA_{i,:}\boldsymbol{\Omega}$. Then, because $\bH= \bA^T \bG= \sum_i{\bA_{i,:}^T}\bG_{i,:}$, the $i$-th term in this summation can be obtained. With all rows of $\bA$, in this way we can accomplish steps 3 and 4 in the  {\randQBfp} algorithm with only one pass over $\bA$. It should be pointed out that this algorithm is not a general single-pass algorithm, as it has the restriction of the matrix. For more general single-pass algorithms for low-rank matrix approximation, please refer to the recent work \cite{Tropp2017}. 
%because steps 3 and 4 can be implemented with only one pass over $\bA$.

\subsection{The inclusion of power iteration scheme} The error of randomized QB factorization could be large for the matrix whose singular value decays slowly \cite{Halko2011}. So, the \emph{power iteration scheme}  has been proposed to relieve this weakness \cite{ Martin2015, Halko2011, Rokhlin2009}. Conceptually, the power iteration means replacing  $\bA$ with $(\bA\bA^T)^P\bA$, where $P$ is an integer. However, in floating-point computation any singular components smaller than $\sigma_1\epsilon_{mach}^{1/(2P+1)}$ will be lost. This makes the orthonormalization steps after the applications of $\bA$ and $\bA^T$ necessary, and $P$ should not be set to a large number. Incorporating the power iteration, we have the {\randQBfp} algorithm for the fixed-precision problem presented as Algorithm 4, where the error indicator $E$ is utilized. 

%\final{Finally we point out a drawback of the \texttt{randQB\_FP} procedure. It is due to the inaccuracy in calculating $\boldsymbol{B}_i$, c.f. (\ref{eq:newBi}). The calulation of $\bB_i$ depends on the equality $\bB^{(i-1)}={\bQ^{(i-1)}}^T\bA$. Any inaccuracy of the equality reflects on the discrepancy between $\bQ_i^T\bA$ and $\bB_i$, causing even larger error in $\bB_{i+1}$ in the next iteration. This error accumulates and prevents the approximation error $\|\bA-\bQ^{(i)}\bB^{(i)}\|$ from continually decreasing. Empirically, the approximation error ceases to decrease at the magnitude of $10^{-7}\|\bA\|$ in many cases.} \st{But this hardly occurs before the limitation of the error indicator emerges. That is, for the applications where we pursue an approximation with moderate relative error, say no less than $10^{-7}$, these numerical drawbacks are negligible. In the following section, we will show numerical results to validate the usage and effectiveness of the proposed algorithms.}
        \begin{algorithm}
        \caption{The \texttt{randQB\_FP} with power scheme for the fixed-precision problem}
        \label{algnew5}
        \begin{algorithmic}[1]
            % SHALL WE RENAME REQUIRE AND ENSURE TO INPUT AND OUTPUT BY \RENEWCOMMAND?
            \REQUIRE $\boldsymbol{A}\in \mathbb{R}^{m\times n}$, desired accuracy tolerance $\varepsilon$, block size $b$, power parameter $P$.
            \ENSURE $\boldsymbol{Q}$, $\boldsymbol{B}$,~such that $\|\boldsymbol{A}- \boldsymbol{QB}\|_{\nF}\leq \varepsilon$.
            \STATE $\boldsymbol{Q} = [~];~  \boldsymbol{B} = [~];$
            \STATE $\boldsymbol{\Omega}= $ randn($n, \tilde{l}$), where $\tilde{l}$ is a sufficiently large number.
           % \STATE $\boldsymbol{H} =\boldsymbol{\Omega}$
            \FOR{$i=1:P$}
            \STATE $\boldsymbol{G} =$ orth($\bA\boldsymbol{\Omega}$)
            \STATE $\boldsymbol{\Omega} =$ orth($\bA^T\boldsymbol{G}$)
            \ENDFOR
        \STATE $\boldsymbol{G} = \bA\boldsymbol{\Omega}$
        \STATE $\boldsymbol{H} = \boldsymbol{A}^T\boldsymbol{G}$
            \STATE $E = \|\boldsymbol{A}\|^2_{\nF}$
            \FOR{$i=1, 2, 3, \cdots$}
            %\WHILE{$\sqrt{E} > \varepsilon$}
            %\STATE $i=i+1$
            \STATE $\boldsymbol{\Omega}_i = \boldsymbol{\Omega}(:,~ (i-1)b+1:ib)$
            \STATE $\boldsymbol{Y}_i = \boldsymbol{G}(:,~ (i-1)b+1:ib) - \bQ(\boldsymbol{B\Omega}_i)$
			\STATE $[\boldsymbol{Q}_i,~\boldsymbol{R}_i]= $ qr($\boldsymbol{Y}_i$)
			\STATE $[\boldsymbol{Q}_i,~\tilde{\boldsymbol{R}_i}]= $ qr($\boldsymbol{Q}_i-\bQ(\bQ^T\boldsymbol{Q}_i)$)
			\STATE $\boldsymbol{R}_i=\tilde{\boldsymbol{R}_i}\boldsymbol{R}_i$
			\STATE $\boldsymbol{B}_i= \boldsymbol{R}_i^{-T}(\boldsymbol{H}(:,~ (i-1)b+1:ib)^T- \boldsymbol{Y}_i^T\bQ\boldsymbol{B} - \boldsymbol{\Omega}_i^T\boldsymbol{B}^T\boldsymbol{B}) $
            \STATE $\boldsymbol{Q} = [\boldsymbol{Q}, ~\boldsymbol{Q}_i]$
            \STATE $\boldsymbol{B} = \left[ \begin{array}{c}
                 \boldsymbol{B}\\
                 \boldsymbol{B}_i
            \end{array}\right]$
                \STATE $E = E - \|\boldsymbol{B}_i\|^2_{\nF}$
			  \STATE \textbf{if} $\sqrt{E} < \varepsilon$ \textbf{then stop}
            \ENDFOR
            %\STATE $k= it$
        
        \end{algorithmic}
        \end{algorithm}

In Algorithm 4, a sufficient large value of $\tilde{l}$ should be set according to problem-specific experience and the concern of computing time. If the set $\tilde{l}$ is not large enough for attaining the specified accuracy criterion, we need to re-generate the $\boldsymbol{\Omega}$ matrix and rerun the algorithm to collect additional columns/rows of $\bQ$ and $\bB$. This situation and the power scheme both increase the number of passes over $\bA$. But compared to other algorithms for the fixed-precision problem, this fixed-precision {\randQBfp} algorithm involves much fewer passes over $\bA$. 

%For the \texttt{randQB\_FP} with power scheme, the number of passes over matrix $\bA$ increases to $2P+1$.\footnote{This can be further reduced because it is possible to apply less frequent orthogonalization in the algorithm \cite{Voronin2015}.} In many scenarios, $P=1$ or $2$ brings sufficient accuracy. For the {\randQBei} algorithm, we can similarly apply the power scheme, just like in \cite{Martin2015}.

% \include{sect5}
% \include{sect6}

\section{Numerical results}
    In this section we compare the proposed algorithms against several existing algorithms in terms of execution time, memory usage and accuracy. All experiments are carried out on a Linux server with two 12-core Intel Xeon E5-2630 CPUs @ 2.30 GHz, and 32GB RAM. For comparison of speed, the proposed algorithms have been implemented in C based on the codes shared by the authors of \cite{Martin2015, RandQR_code}. The program is coded with OpenMP derivatives, and compiled with the Intel ICC compiler with MKL libraries \cite{ICC}, to take full advantage of the multi-core CPUs. 
The QR factorization and other basic linear algebra operations are implemented through LAPACK routines which are automatically executed in parallel.
    \subsection{Comparison of speed} 
        We compute the QB factorization of an $n\times n$ matrix $\bA$. Notice the singular value distribution of matrix is immaterial for this runtime comparison. Four algorithms are compared:
\begin{itemize}
            \item The \texttt{randQB} algorithm in Figure 1(a);
            \item The \texttt{randQB\_b} algorithm in Figure 1(b), obtained from \cite{RandQR_code};
            \item The \texttt{randQB\_EI} algorithm presented in Section 3; 
            \item The \texttt{randQB\_FP} algorithm presented in Section 4.
\end{itemize}
We compare their speed using both dense and sparse matrices, both as a function of the dimension of the matrix and the parameter $l$ denoting the number of the output $\bQ$'s columns.
The block size is $b=20$ for the \texttt{randQB\_b}, \texttt{randQB\_EI} and \texttt{randQB\_FP} algorithms.   
For each runtime measurement, the average time over 20 runs is reported. Notice that the compared {\randQBb} algorithm is an efficient parallel  implementation open-sourced in \cite{RandQR_code}, also based on Intel MKL libraries. 
% We did not test the partial QR factorization, because the experiments in \cite{Martin2015} showed that the block randQB algorithm is nearly 10X faster than it. 
        
In the first experiment we test the algorithms on dense matrices of varying size. $n$ ranges from 2,000 to 40,000. The value of $l$ is always 200. The results are shown in Figure 3 for the situations without and with the power scheme. The data of the blocked randQB algorithm for the matrix with $n=40,000$ are not available due to unreasonably long runtime of the program from \cite{RandQR_code}. %\footnote{In \cite{Martin2015}, the largest matrix in the experiment of blocked randQB algorithm is just of $16,000\times 16,000$. }
From the results in Figure 3(a), we see that the \texttt{randQB\_EI} and \texttt{randQB\_FP} algorithms are 2.4X (13.47s vs. 31.78s) and 4.0X (8.01s vs. 31.78s) faster than the implementation of {\randQBb} 
\begin{figure*}[ht]
        \centering
\subfigure[Without the power scheme] {\includegraphics[width=2.54in]{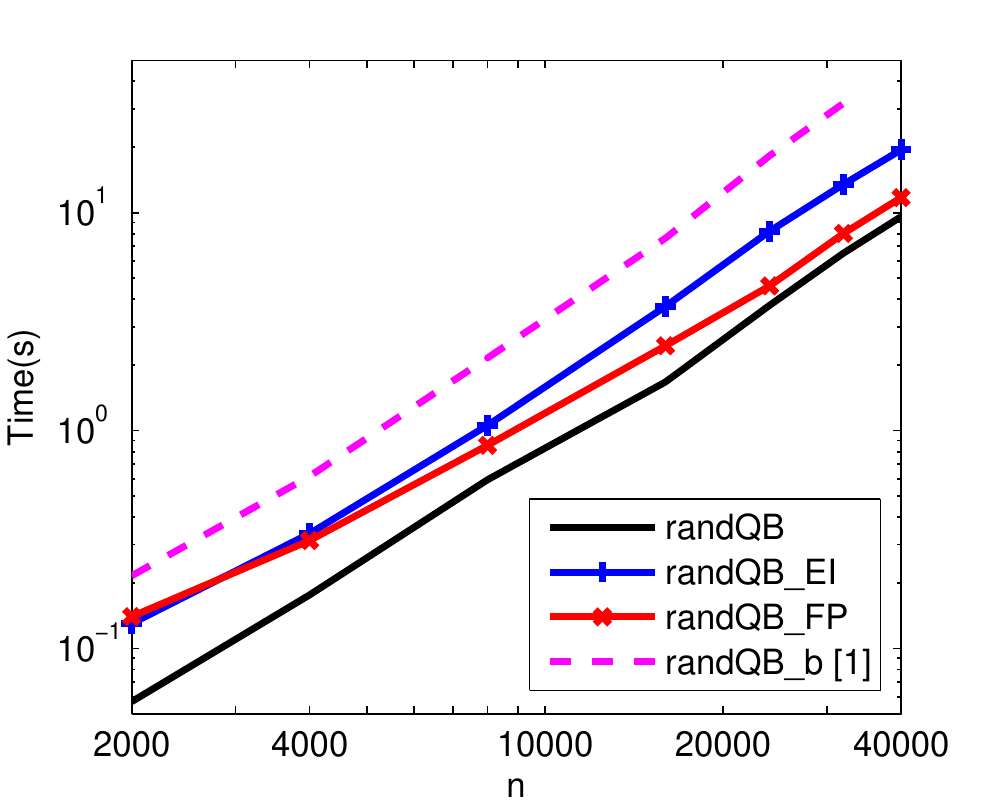}}
\subfigure[With the power scheme] {\includegraphics[width=2.54in]{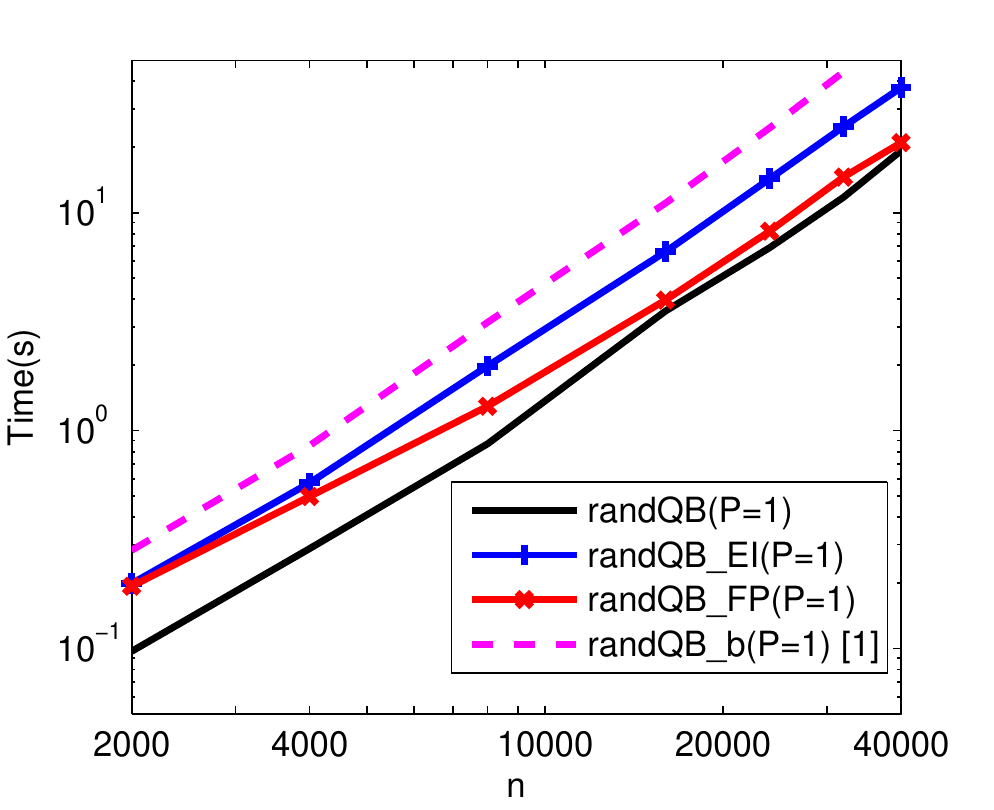}}
        \caption{Runtime of the randomized QB-factorization algorithms for dense matrices ($l=200$).}
        \label {test1p1}
        \end{figure*}    
algorithm respectively, when $n=32,000$. If the power scheme is imposed, the acceleration ratios decrease to 1.8X and 3.0X respectively, which are still remarkable. The \texttt{randQB} algorithm has the fastest computational speed, but its advantage over \texttt{randQB\_FP} algorithm becomes marginal when the matrix size is large.  
        
Here we only show the runtime results with the power parameter $P=1$, as for many applications this already achieves sufficient accuracy.
        
        The memory costs for some large matrices are listed in Table 2. For the \texttt{randQB}, \texttt{randQB\_EI} and \texttt{randQB\_FP} algorithms, the memory cost is mainly due to storing matrix $\bA$. For the \texttt{randQB\_b} algorithm, it needs additional memory to store the residual matrix and the product of $\boldsymbol{QB}$. So, the proposed algorithms consume about 1/3 of that used by the blocked randQB algorithm. If we allow that $\bA$ can be overwritten by the residual matrix, the memory cost of the  \texttt{randQB\_b} algorithm \cite{Martin2015} can be reduced, but still 2X larger than the proposed algorithms. 
\begin{table}[h]
  \caption{The memory usage of the randomized QB-factorization algorithms for dense matrices ($l=200$)}
  \label{tab:table1}
  \centering
\small{
\begin{spacing}{1}
  \begin{tabular}{*{5}{c}} \toprule
  $n$ & \texttt{randQB} & \texttt{randQB\_EI} &  \texttt{randQB\_FP} &  \texttt{randQB\_b}~[1] \\ \midrule 
% 2000 & 101 MB &  65 MB &  72 MB &  132 MB \\  
%4000 & 221 MB &  167 MB &  188 MB &  429 MB \\ \hline 
% 8000 & 749 MB &  570 MB &  607 MB &  1.5 GB \\ 
 16,000 & 2,308 MB &  2,303 MB &  2,357 MB &  6,237 MB \\  
 24,000 & 4,792 MB &  4,796 MB &  4,873 MB &  13,618 MB \\  
 32,000 & 8,253 MB &  8,253 MB &  8,356 MB &  23,931 MB \\  
 40,000 & 12,694 MB &  12,581 MB &  12,714 MB &  N.A. \\ %\midrule 
%  $n$ & \texttt{randQB}($P$=1) & \texttt{randQB\_EI}($P$=1) &  \texttt{randQB\_FP}($P$=1) &  \texttt{randQB\_b}($P$=1) [1] \\ \midrule 
% 16000 & 2.1 GB &  2.0 GB &  2.1 GB & 3.9 GB \\ 
% 24000 & 4.5 GB &  4.5 GB &  4.6 GB &  8.4 GB \\ 
% 32000 & 7.9 GB &  7.8 GB &  7.9 GB &  15 GB \\  
% 40000 & 12 GB &  12 GB &  12 GB &  -- 
\bottomrule 
 \end{tabular}
 \end{spacing}
 }
\end{table}
For the largest case with $n=40,000$, the \texttt{randQB\_b} algorithm actually requests more memory than the size of RAM ($\sim$ 32 GB), which explains the aforementioned long runtime of \texttt{randQB\_b}. 
%\begin{table}[tbhp] %t-top, b-bottom, h-here, p-page of float figure/table/algorithms

The second experiment is about the algorithms' efficiency for sparse matrices. We generate sparse matrices with roughly 0.3\% non-zero elements. They are stored in CSR (compressed sparse row) format \cite{SaadBook}. 
The runtimes of the algorithms are shown in Figure 4. The results of the \texttt{randQB\_b} algorithm for the matrices with $n\ge40,000$ are not available due to unreasonably long runtime. In contrast, it only takes a couple of seconds for the other algorithms to process the largest matrix with $n=48,000$. We see that the proposed algorithms take usage of the sparsity, while the blocked randQB algorithm cannot. The speedup ratios of the former to the latter increase as the matrix size increases. For $n=32,000$, the \texttt{randQB\_EI} and   \texttt{randQB\_FP} algorithms are more than 22X and 14X faster than the implementation of 
\begin{figure*}[h]         \centering
\subfigure[Without the power scheme] {\includegraphics[width=2.54in]{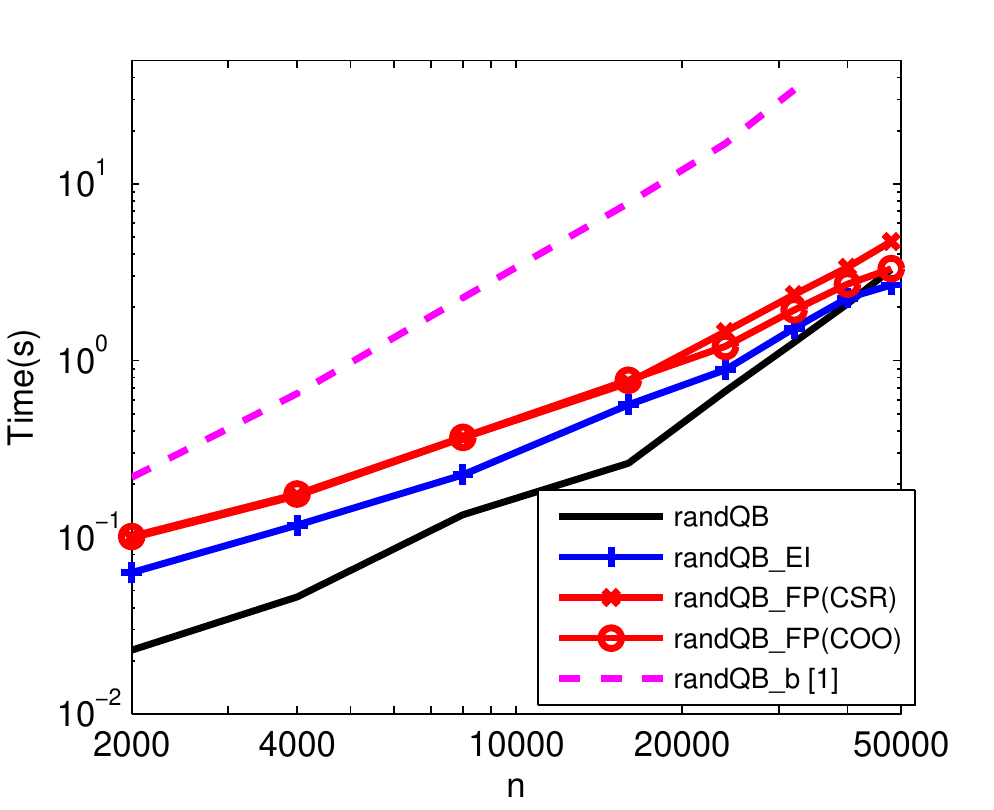}}
\subfigure[With the power scheme] {\includegraphics[width=2.54in]{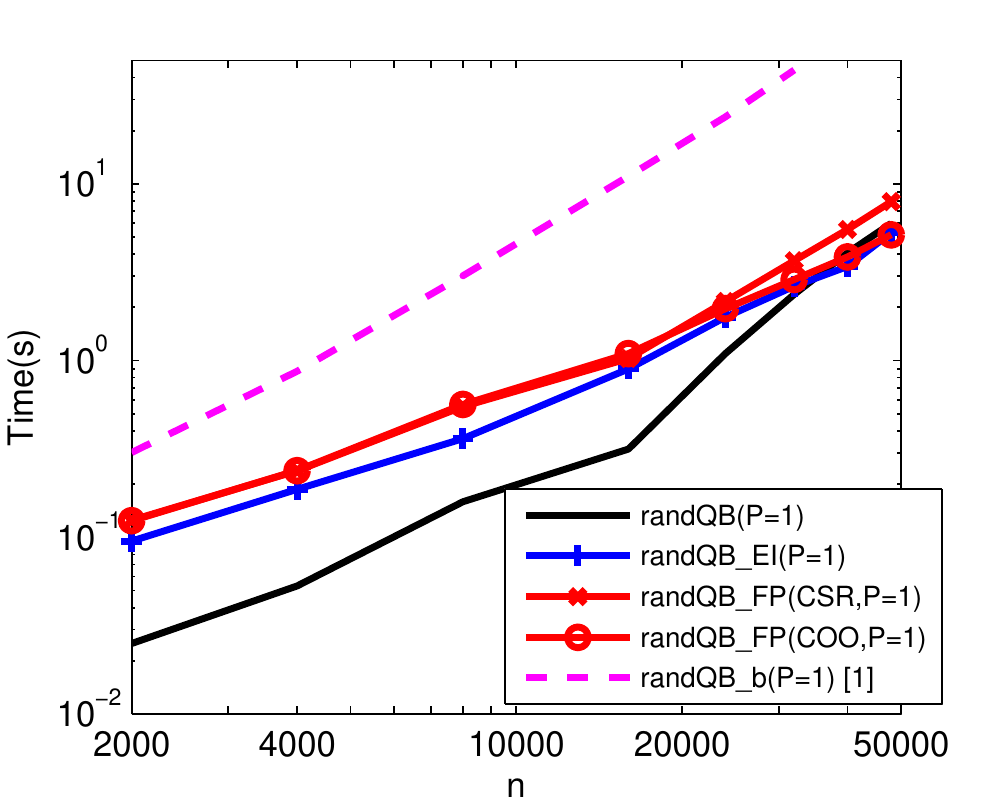}}
        \caption{Runtime of the randomized QB-factorization algorithms for sparse matrices ($l = 200$).}
        \label {test2}
        \end{figure*}
\texttt{randQB\_b} algorithm, respectively. Different from the situation for dense matrices, the \texttt{randQB\_EI} algorithm becomes faster than \texttt{randQB\_FP}. This implies that lumping the multiplications of a sparse matrix all together brings less benefit. 
And, \texttt{randQB\_EI} could run faster than the \texttt{randQB} algorithm for matrix size over 48,000. 
This can be explained by the comparison of (10) and (12), the inefficiency of orthogonalizing the whole matrix of $l$ columns, and that the sparse matrix removes the benefit of BLAS-3 operation to \texttt{randQB}. 
Another interesting phenomenon is that if we instead store a large sparse matrix with the COO (coordinate) format, the runtime of \texttt{randQB\_FP} algorithm can be reduced by 30\%. This means that the COO format is more adaptive to parallel computing. If we set $P=1$, similar observations regarding the experimental results can be drawn, as shown in Figure 4(b).

The memory cost of these algorithms are 
% Besides, the \texttt{randQB\_FP} with COO format of sparse matrix consumes a couple of percents more memory, which is not listed in the table. 
%\begin{table}[tbhp] %t-top, b-bottom, h-here, p-page of float figure/table/algorithms
 listed in Table 3, from which we see more prominent memory saving of the proposed algorithms over the \texttt{randQB\_b} algorithm. While compared with \texttt{randQB}, the proposed algorithms consume comparable memory.

        Lastly, we test a dense matrix with size $n=8,000$, and vary the value of $l$. The trends of the runtime are plotted in Figure 5. It shows that the \texttt{randQB\_EI} and \texttt{randQB\_FP} algorithms without the power scheme
 are about 1.9X and 2.5X faster than \texttt{randQB\_b}, respectively. If the power scheme is imposed, the speedup ratios to \texttt{randQB\_b} decrease, but \texttt{randQB\_FP} is still more than 2X faster than \texttt{randQB\_b}.
\begin{table}[tbh]
  \caption{The memory usage of the randomized QB-factorization algorithms for sparse matrices ($l=200$)}
  \label{tab:table2}
  \centering
\small{
\begin{spacing}{1}
  \begin{tabular}{*{5}{c}} \toprule
  $n$ & \texttt{randQB} & \texttt{randQB\_EI} &  \texttt{randQB\_FP} &  \texttt{randQB\_b} [1] \\ \midrule 
% 8000 & 213 MB &  81 MB &  128 MB &  1.5 GB \\  
 16,000 & 162 MB &  174 MB &  223 MB &  6,153 MB \\ 
 24,000 & 232 MB &  239 MB &  312 MB &  13,572 MB \\  
 32,000 & 293 MB &  303 MB &  402 MB &  23,917 MB \\  
 40,000 & 338 MB &  359 MB &  488 MB &  N.A. \\ 
  48,000 & 405 MB & 426 MB &  582 MB &  N.A. \\ %\midrule 
%  $n$ & \texttt{randQB}($P$=1) & \texttt{randQB\_EI}($P$=1) &  \texttt{randQB\_FP}($P$=1) &  \texttt{randQB\_b}($P$=1) [1] \\ \midrule 
% 16000 & 220 MB &  153 MB &  316 MB & 3.9 GB \\  
% 24000 & 228 MB &  220 MB &  445 MB &  8.4 GB \\  
% 32000 & 282 MB &  322 MB &  501 MB &  15 GB \\  
% 40000 & 305 MB &  411 MB &  590 MB &  -- \\ 
%   48000 & 389 MB &  478 MB & 664 MB &  -- \\ 
\bottomrule 
 \end{tabular}
 \end{spacing}
 }
\end{table}
\begin{figure*}[th]
        \centering
\subfigure[Without the power scheme] {\includegraphics[height=2.06in]{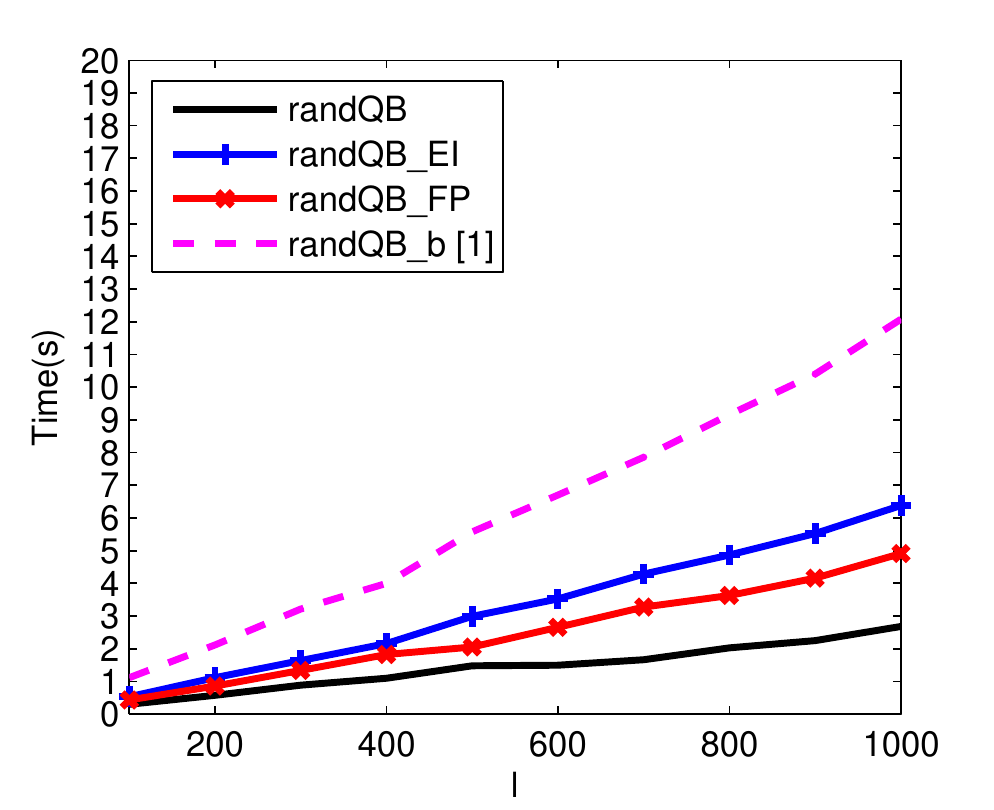}}
\hspace{-6pt}
\subfigure[With the power scheme] {\includegraphics[height=2.06in]{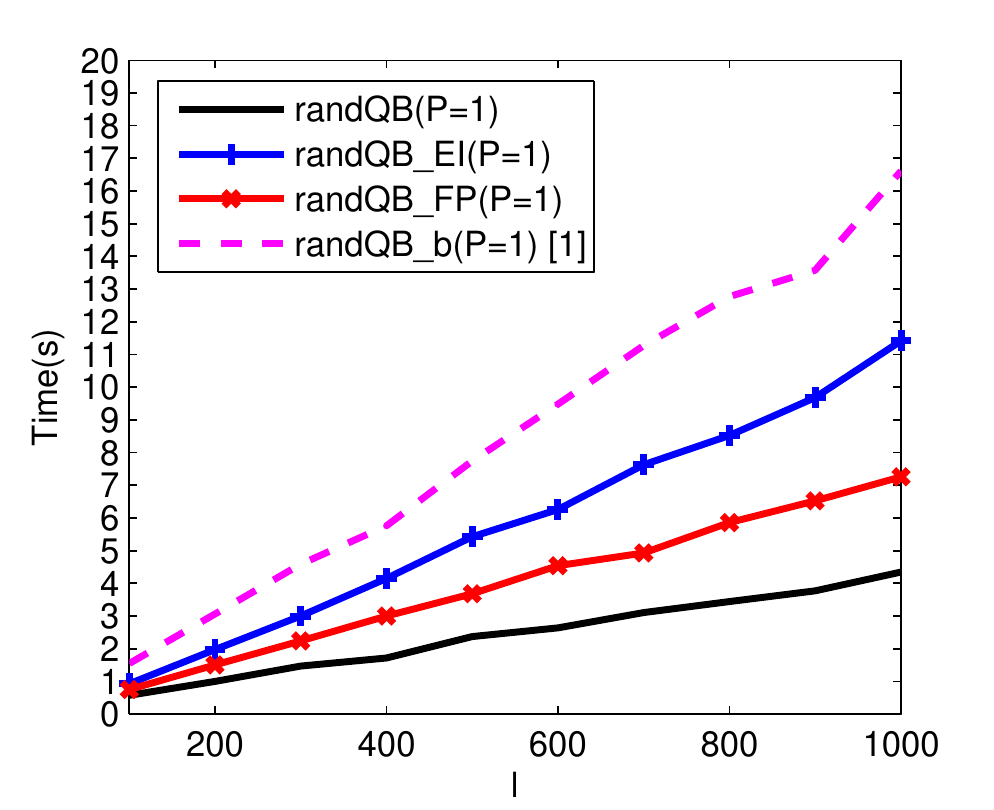}}
        \caption{Runtime of the randomized QB-factorization algorithms for a dense matrix ($n=8,000$) with varying value of $l$.}
        \label {test3eps}
        \end{figure*}

    % Test the 5 algorithms on two types of matrices, whose sizes vary, and give plot of timing result
        % execution time of 1 - 4 (with p=0,1,2), where fixed rank k = 200
            % versus
                % size of matrices: 1, 2, 4, 8, 16, 32 (k)
    % in both dense and sparse cases.
    % (k=200, dense), (k=200, sparse)
    % 2 plots in total.
    
    % Test each algorithm with different k but fix n
        % timings of 1-4 with p=0,1,2 where n=8000
            % versus
                % k=20,40,...800
    % in both dense and sparse case
    % (n=8000, dense), (n=8000, sparse)
    % 2 plots in total
    
    % raw data: 4 columns are: basicQB / martin / algo 1 / algo 3.

    \subsection{Comparison of accuracy}
    Three kinds of matrices are tested standing for different distribution patterns of singular values:
    \begin{itemize}
        \item \textbf{Matrix 1 (slow decay)}: $\bA = \boldsymbol{U\Sigma V}$, where $\boldsymbol{U}$ and  $\boldsymbol{V}$ are randomly drawn matrices with orthonormal columns, and the diagonal matrix $\boldsymbol{\Sigma}$ has diagonal elements $\sigma_{j} = 1/j^2$.
        \item \textbf{Matrix 2 (fast decay)}: $\bA$ is formed just like Matrix 1, but the diagonal elements of $\boldsymbol{\Sigma}$ is given by $\sigma_{j} = e^{-j/7}$. It reflects a fast decay of singular values.
        \item \textbf{Matrix 3 (S-shape decay)}: $\bA$ is built in the same manner as Matrix 1 and Matrix 2, but the diagonal elements of $\boldsymbol{\Sigma}$ are given by $\sigma_{j} = 0.0001+ (1+e^{j-30})^{-1}$. It makes the singular values first hover around 1, then decay rapidly, and finally level out at about 0.0001.
    \end{itemize}    
For each kind, we generate a $2,000\times 2,000$ matrix, for which we compare the errors of the proposed techniques and the blocked randQB scheme \cite{Martin2015} for varying $l$ values. The results are shown in Figure 6, where we see that the     \begin{figure*}[h]
\centering
\subfigure[Errors of Matrix 1 whose singular values decay slowly] {\includegraphics[width=3.1in]{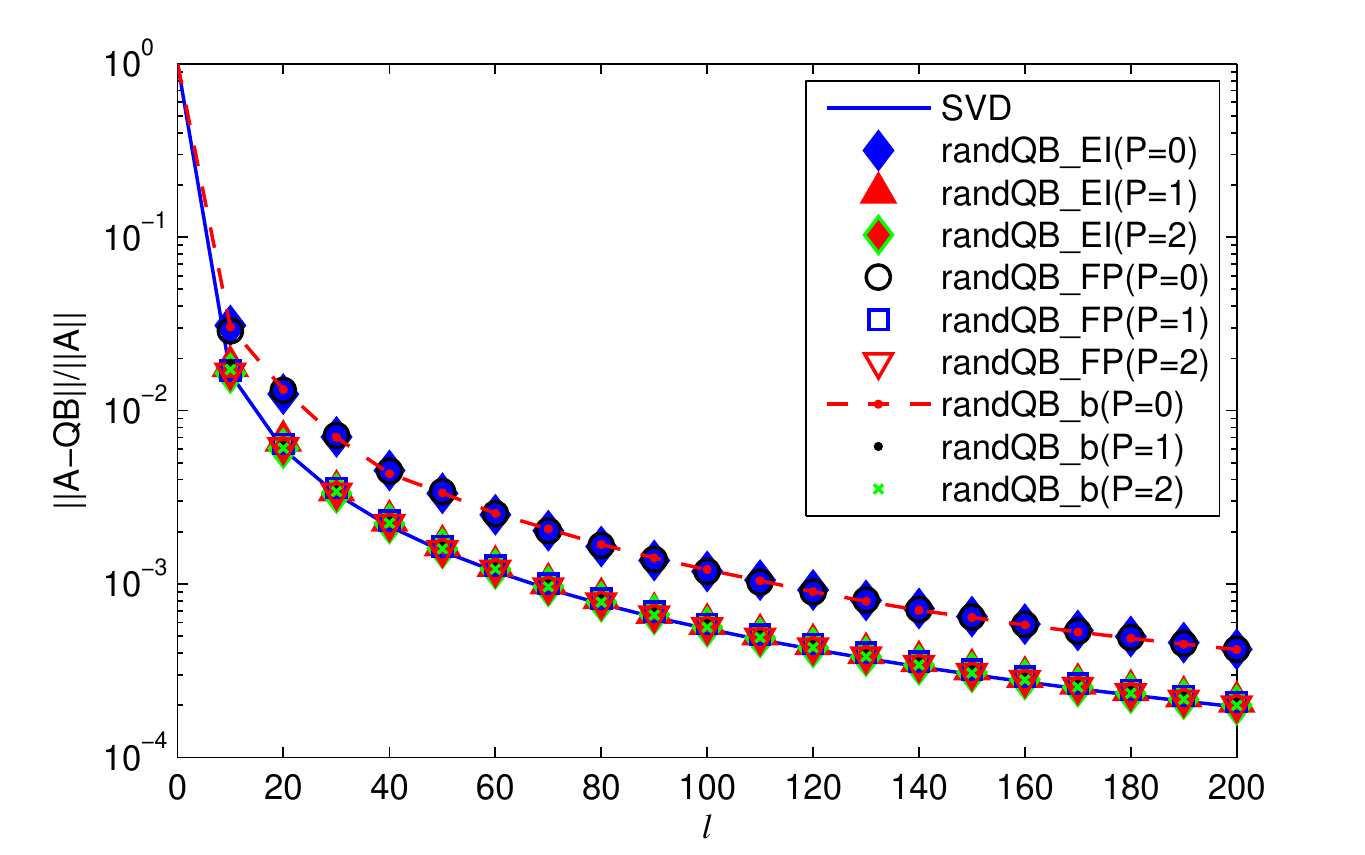}}
\par\vspace{-8pt}
\subfigure[Errors of Matrix 2 whose singular values decay rapidly] {\includegraphics[width=3.1in]{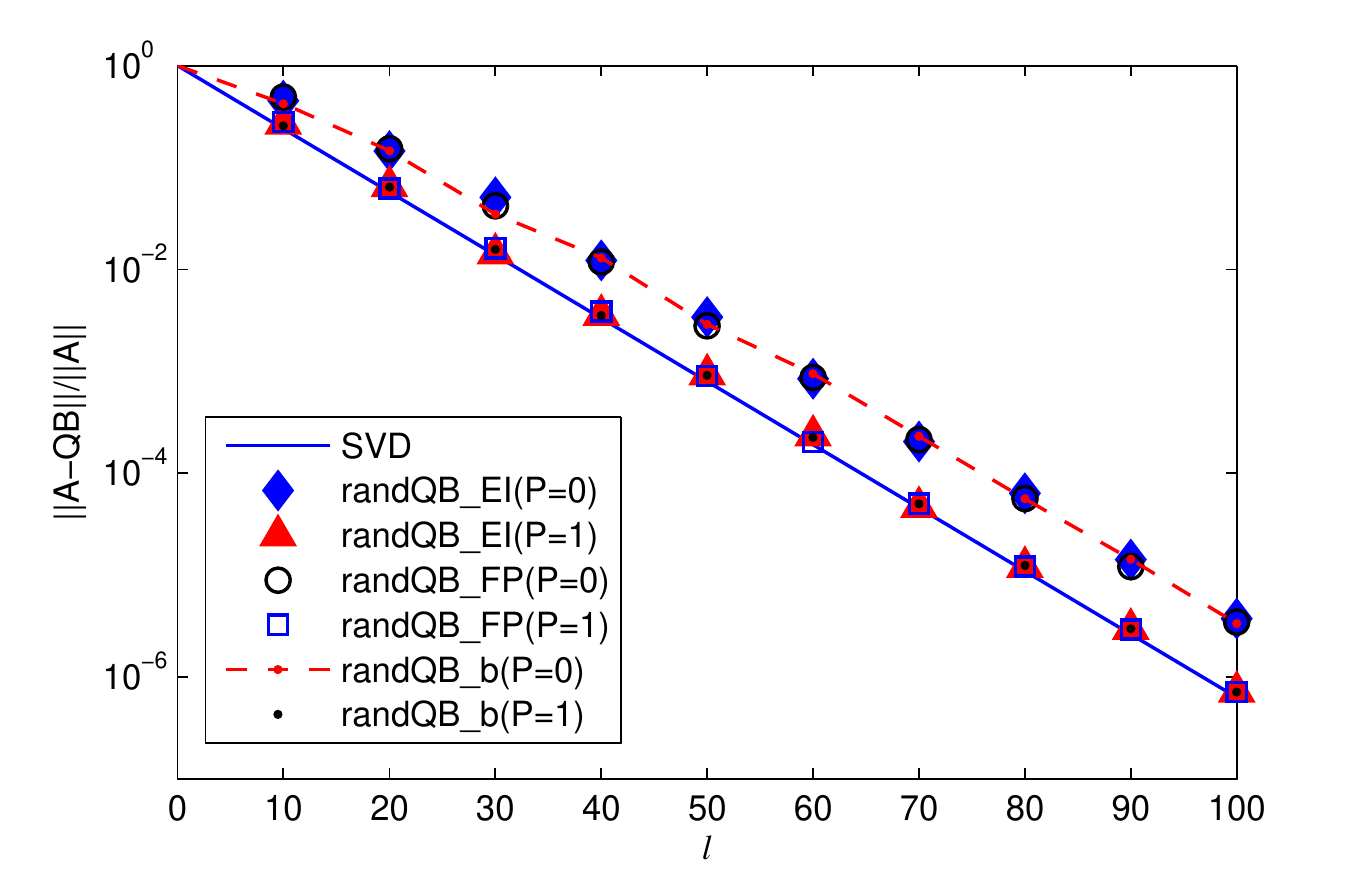}}
\par\vspace{-8pt}
\subfigure[Errors of Matrix 3 with S-shape decay of singular values] {\includegraphics[width=3.1in]{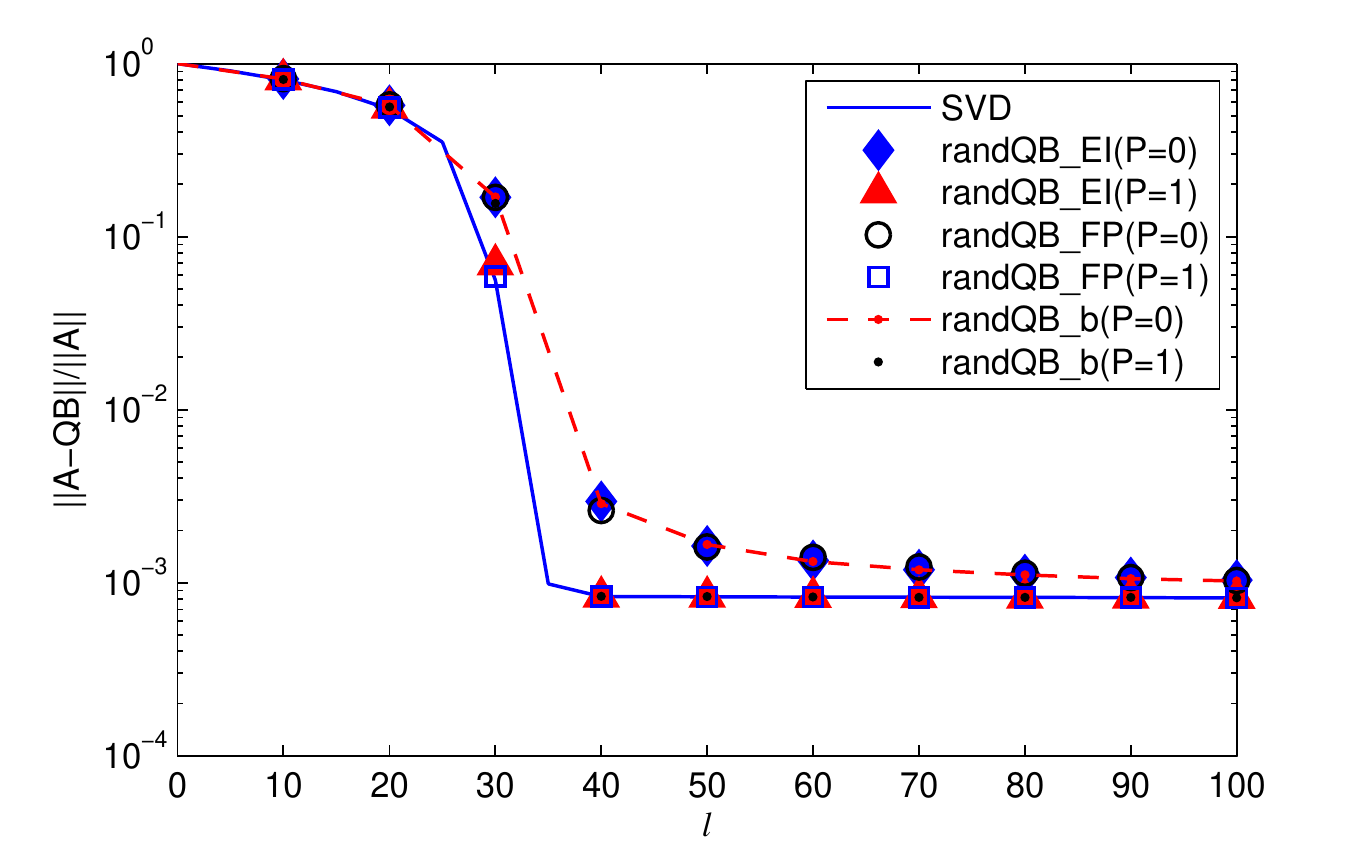}}
\caption{Errors of different algorithms for approximating the test matrices ($n=2,000$, 
$b=10$).}
\label {acc_figures}
\end{figure*}
proposed techniques have just the same accuracy as the blocked randQB algorithm. 
%    \begin{figure*}[h]
%    \centering
%    \includegraphics[height=2in]{accplot4v3.eps}
%    \caption{The relative errors of the QB approximation and the corresponding estimations obtained from the error indicator in the \texttt{randQB\_EI} and \texttt{randQB\_FP} Algorithms for Matrix 2. %The error indicator fails to work if the rank $l$ increases causing the relative error smaller than $5\times 10^{-7}$.
%    }
%    \label {indicator}
%    \end{figure*}
%They are  mathematically equivalent after all.
If we use the power scheme, even with a power parameter 
as small as $P=1$, the errors of the \texttt{randQB\_EI} and \texttt{randQB\_FP} algorithms are remarkably reduced. 
And, the power schemes with $P=1$ and $P=2$ produce indistinguishable results for the tested matrices. Both are extremely close to the optimal results from SVD.
%In the plots for Matrix 2 and 3 we do not show the results of the power scheme with $P=2$. 
%The data shown in Figure 6 refer to a single instantiation of the randomized algorithm; but they have already shown sufficient accuracy.

%    The validity of the error indicator proposed in 
% Section 3 is critical for solving the fixed-precision problem. To evaluate its accuracy, we have tested Matrix 2, and drawn the relative errors of approximation and the corresponding results derived from the error indicator for different $l$ values in Figure 6. From the figure we see that, if the actual error is no less than $5\times 10^{-7}$, the results from the error indicator in Algorithm 1 and Algorithm 4 are very accurate. When it becomes smaller (for larger $l$), the results from the error indicator stagnate, failing to reflect the actual error. This result verifies the analysis given at the end of Section 3.    
    
   \subsection{Performance of the single-pass algorithm}
    Without the power scheme, the \texttt{randQB\_FP} algorithm is a single-pass algorithm (see Algorithm 3). This is because $\boldsymbol{G}=\bA\boldsymbol{\Omega}$ and $\boldsymbol{H}=\bA^T \boldsymbol{G}$ can be executed through one pass over matrix $\bA$, providing that $\bA$ is in the row-major format. 
 Another single-pass algorithm was proposed in \cite{Halko2011}, as a remedy to the \texttt{randQB} algorithm. It is shown in Figure 7, whose step (3) produces matrices $\bQ$ and $\tilde{\bQ}$ such that $\bA \approx \bQ \bQ^T\bA \tilde{\bQ} \tilde{\bQ}^T$. Then, a small matrix $\hat{\bB}= \bQ^T\bA \tilde{\bQ}$ is approximately solved in step (4), because $\tilde{\bQ}^T \tilde{\boldsymbol{Y}} \approx \tilde{\bQ}^T \bA^T \bQ \bQ^T \tilde{\boldsymbol{\Omega}}= \hat{\bB}^T \bQ^T \tilde{\boldsymbol{\Omega}}$. This single-pass algorithm corresponds to the low-rank factorization in form of $\bA \approx \bQ \hat{\bB} \tilde{\bQ}^T$. Obviously, it includes more approximations and is not equivalent to the {\randQB} algorithm. In contrast, Algorithm 3 is mathematically equivalent to {\randQB} (see Proposition 2), and is supposed to be more accurate. With Matrix 1 and Matrix 2 from Sect. 5.2, we can compare the accuracy of the both algorithms, whose results are plotted in Figure 8.    
\begin{figure*}[hbt]
\centering
\includegraphics[height=1.2in]{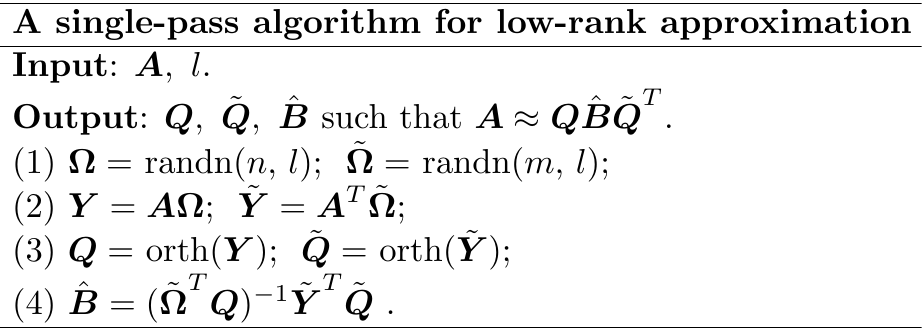}
\caption{An existing single-pass algorithm for low-rank matrix factorization \cite{Halko2011}.}
\label {fig8}
 \end{figure*}

\begin{figure*}[hbt]
\centering
\includegraphics[height=2.3in]{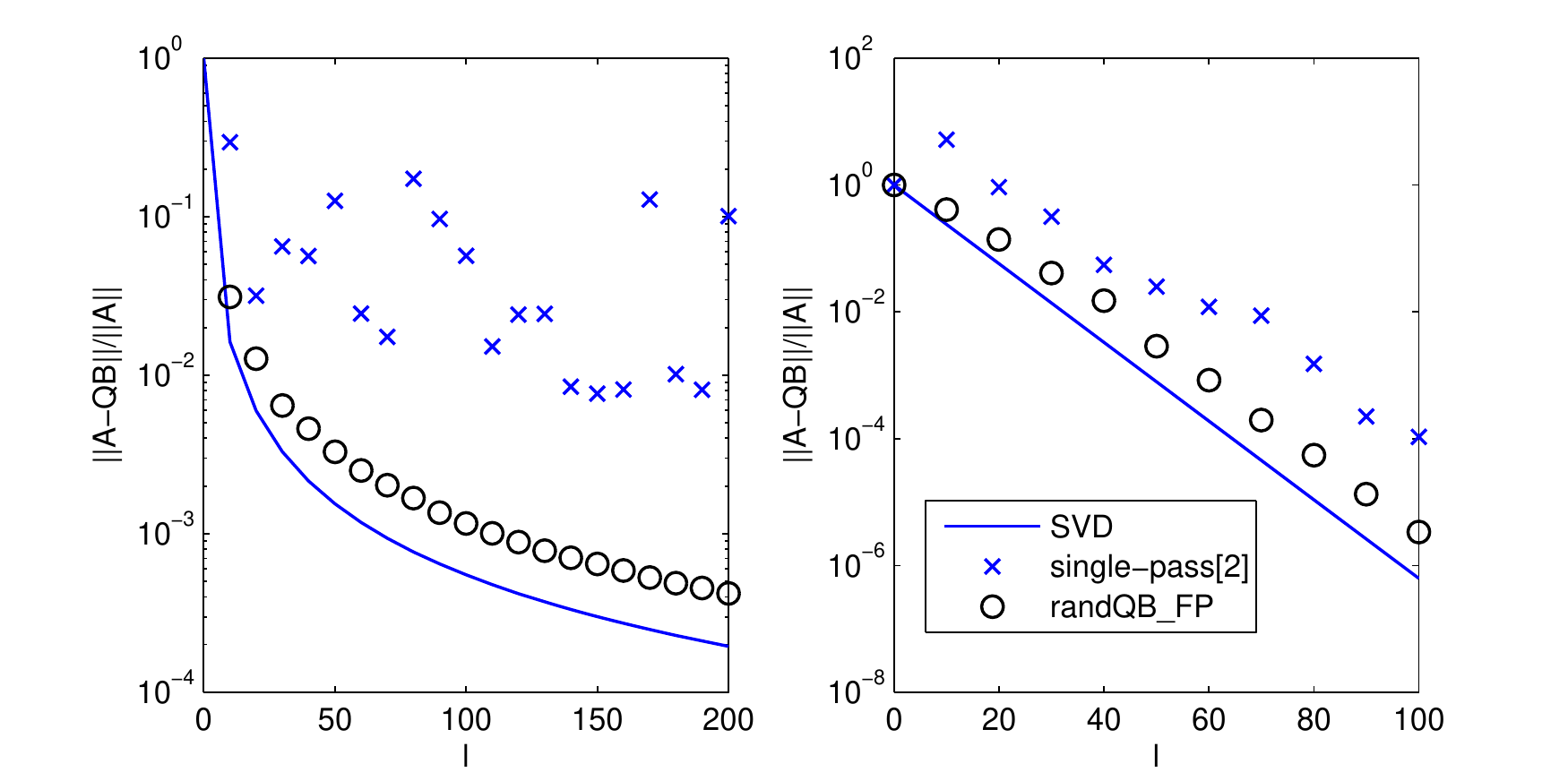}
\caption{The approximation errors of the two single-pass algorithms and truncated SVD for Matrix 1 (left) and Matrix 2 (right).}
\label {fig9}
 \end{figure*}

From Figure 8 we see that the approximation error of the single-pass algorithm in \cite{Halko2011} is often one order of magnitude larger than that of our \texttt{randQB\_FP} based algorithm. Actually, it does not even decrease as the rank of the approximation matrix increases. We also calculate the top 50 singular values, and the over-sampling with $s=10$ is applied to the both algorithms. The results are shown in Figure 9, along with those obtained from the \texttt{randQB} algorithm, where the results of \texttt{randQB\_FP} and \texttt{randQB} are indistinguishable. For the matrix with slow decay of singular value the result from \texttt{randQB\_FP} shows moderate accuracy on the top singular values, whose error is usually orders of magnitude smaller than that of the single-pass algorithm in \cite{Halko2011}.
\begin{figure*}[h]
\centering
\includegraphics[height=1.8in]{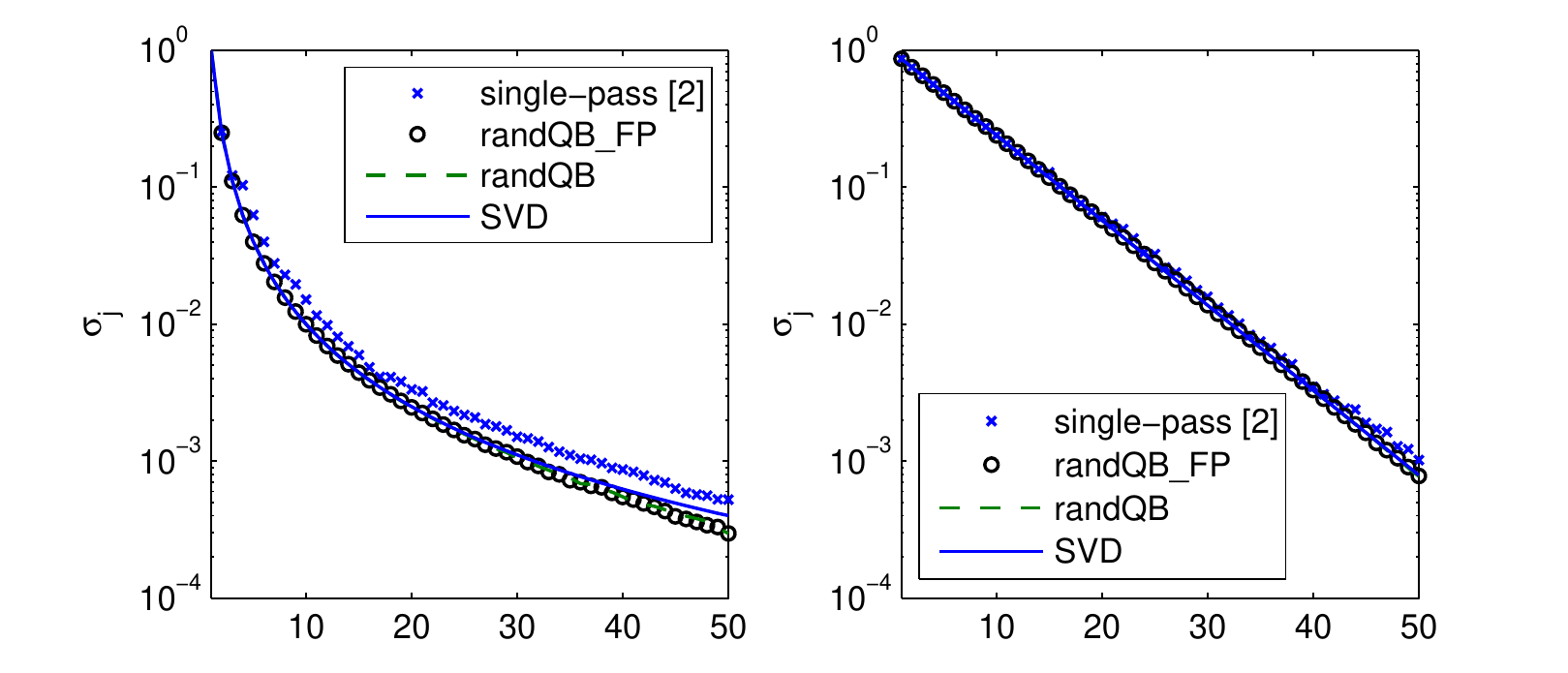}
\caption{The computational results of top singular values for Matrix 1 (left) and Matrix 2 (right).}
\label {fig10}
 \end{figure*}

    \subsection{Results of solving the fixed-precision problems}
    In this subsection we test the proposed algorithms with some fixed-precision problems. The optimal solution is the factorization with the smallest sizes of $\bQ$ and $\bB$, which corresponds to less amount of subsequent computation. It can be achieved by first calculating SVD of the input matrix $\bA$, and then checking $(\sum_{j=k+1}^{\min(m,n)} \sigma_j^2)^{1/2}$, where $\sigma_j$ is $\bA$'s $j$-th singular value, to determine the smallest $k$ satisfying the accuracy criterion. Here, we always consider the accuracy criterion with a relative tolerance: $\| \boldsymbol{A} - \boldsymbol{QB} \|_{\nF} <\epsilon \|\bA\|_{\nF}$.
    
The proposed algorithms are compared with the SVD based method and the \emph{adaptive randomized range finder} (Algorithm 4.2 in [2]). The row-by-row calculation scheme mentioned in Section 3 is implemented into our algorithms. The experiments are carried out with Matlab on the aforementioned Linux server. The built-in commands like ``\texttt{svd}'', ``\texttt{qr}'', etc. are employed, which naturally take advantage of parallel computing.%\footnote{Seems the Matlab routines only take usage of one CPU, i.e., 12 cores.} 

For the \texttt{randQB\_FP} algorithm, we empirically set $\tilde{l}=50b$. 
%This limits the number of iterations and therefore avoids the reliability issue of the \texttt{randQB\_FP} algorithm. 
With a suitable value of $b$, this produces a large enough $\tilde{l}$ for attaining the accuracy criteria in the experiments. A more sophisticated  approach for setting $\tilde{l}$ and $b$ can be investigated in the future.  
    
   We first construct the three kinds of matrices in Sect. 5.2, each of $8,000 \times 8,000$ size, and test them with the four methods. Their results are shown in Table 4. For \texttt{randQB\_EI} and \texttt{randQB\_FP}, the power scheme with $P=1$ is used. The block size is set to $b=10$ in all tests, except the last one for which $b=40$. In Table 4, ``$\epsilon$'' stands for the threshold for relative error, and ``error'' means the relative Frobenius-norm error of the produced QB factorization. From the table we see that the results of \texttt{randQB\_EI} and \texttt{randQB\_FP} algorithms all satisfy the set accuracy demands. And, the corresponding ranks (i.e., the number of columns in $\bQ$) are very close to the the optimal values from the SVD 
based approach. As for the runtime, the proposed algorithms are usually several tens times faster than SVD. Notice that our Matlab programs are less optimized than the built-in \texttt{svd} command. So, more significant speedup could be expected for the implementation in C. 
\begin{table}[h]
  \caption{The results on solving the fixed-precision problems based on three test matrices}
  \label{tab:table3}
  \centering
\small{
\begin{spacing}{0.95}
\renewcommand{\multirowsetup}{\centering}
%  \begin{tabular}{|c c c c c c c c c c c|} 
  \begin{tabular}{*{11}{c}} 
  \toprule
Matrix 
 & \multirow{2}{*}{$\epsilon$} & \multicolumn{3}{c}{randQB\_EI} &  \multicolumn{2}{c}{randQB\_FP} &  \multicolumn{2}{c}{truncated SVD} &  \multicolumn{2}{c}{RangeFinder [2]} \\ \cmidrule(r){3-5} \cmidrule(r){6-7} \cmidrule(r){8-9} \cmidrule(r){10-11}
 type &  & rank & time(s) & error & rank & time(s) & rank & time(s) & rank & time(s) \\ \midrule
\multirow{2}{*}{Matrix 1} & 1e-2 & \textbf{15} & 1.19 & 9.3e-3 & \textbf{15} & 3.13 & \textbf{15} & \multirow{2}{*}{123} & 115 & 1.1 \\
 & 1e-4 &\textbf{327} & 8.29 & 9.98e-5 & \textbf{328} & 3.71 & \textbf{313} & & 2,084 & 29.1 \\ \midrule
\multirow{2}{*}{Matrix 2} & 1e-4 & \textbf{66} & 2.16 & 8.37e-5 & \textbf{66} & 2.73 & \textbf{65} & \multirow{2}{*}{115} & 101 & 1.1 \\
 & 1e-5 &\textbf{82} & 2.68 & 8.86e-6 & \textbf{82} & 3.17 & \textbf{81} & & 113 & 1.1 \\ \midrule
\multirow{2}{*}{Matrix 3} & 1e-2 & \textbf{33} & 1.56 & 4.1e-3 & \textbf{33} & 3.62 & \textbf{32} & \multirow{2}{*}{126} & 3,618 & 87.8 \\
 & 1.5e-3 &\textbf{1,588} & 18.7 & 1.499e-3 & \textbf{1,587} & 15.8 & \textbf{1,587} & & 7,916 & 379 \\ 
\bottomrule 
 \end{tabular}
 \end{spacing}
 }
\end{table}   
Although the adaptive range finder is built on a theory with spectral norm of matrix, in our experiments it always produces a QB factorization satisfying the accuracy demand in Frobenius norm. However, its results (factor matrices) are much larger than necessary.
 
   We then test the algorithms with two real data. One is from a scenic image \cite{image}, and the other is from an information retrieval application ``AMiner'' \cite{AMiner}. The colored image is represented by a %$3168\times 4752 \times 3$ array which is reorganized into a 
   $9,504\times 4,752$ matrix. The other is an $8,130\times 100,000$ keyword-person matrix produced with the term frequency and inverse document frequency (TF-IDF) model \cite{Matbook2007}. This sparse matrix has about 0.2\% nonzero elements. The computational results are listed in Table 5, with different power parameters and block sizes. They again validate that the proposed algorithms can automatically satisfy the accuracy criterion. And, with $P=2$ the result of rank is substantially reduced, approaching the optimal value. For the same power scheme, setting larger block size $b$ we can reduce the runtime of \texttt{randQB\_EI}. In contrast, the runtime of \texttt{randQB\_FP} increases with the block size, as we have set $\tilde{l}=50b$.
Notice that with the relative error tolerance $\epsilon=0.1$, the image is largely compressed ($\sim$ 7X size reduction), with little loss of quality (see Figure 10). 
And, the singular value of ``AMiner'' matrix decays very slowly,
but even with large approximation error its low-rank approximation could bring improved performance of information retrieval (c.f. \cite{Matbook2007}, Sect. 11.3). 
\begin{table}[tbhp]
  \caption{The results on solving the fixed-precision problems based on two real data}
  \label{tab:table4}
  \centering
\small{
\begin{spacing}{0.95}
\renewcommand{\multirowsetup}{\centering}
  \begin{tabular}{@{}c@{~}c@{~}c@{~}c@{}c@{}c@{~}c@{}c@{}c@{}c@{}c@{}c@{}} 
%  \begin{tabular}{*{12}{c}} 
  \toprule
\multirow{2}{*}{Matrix}
 & \multirow{2}{*}{$\epsilon$} & \multirow{2}{*}{parameters} & \multicolumn{3}{c}{randQB\_EI} &  \multicolumn{2}{c}{randQB\_FP} &  \multicolumn{2}{c}{truncated SVD} &  \multicolumn{2}{c}{RangeFinder [2]} \\ \cmidrule(r){4-6} \cmidrule(r){7-8} \cmidrule(r){9-10} \cmidrule(r){11-12}
  &  & & rank & time(s) & error & rank & time(s) & rank & time(s) & rank & time(s) \\ \midrule
\multirow{4}{*}{image} & \multirow{4}{*}{0.1} & $P$=1, $b$=10 & 468 & 8.1 & 0.0999 & 471 & 3.25 & \multirow{4}{*}{\textbf{426}} & \multirow{4}{*}{44.2} & \multirow{4}{*}{2,913} & \multirow{4}{*}{79.0} \\
 & & $P$=1, $b$=20 & 468 & 4.23 & 0.0999 & 472 & 4.44 &  &  &  &  \\ \cmidrule(r){3-8}
 & & $P$=2, $b$=10 & \textbf{441} & 9.98 & 0.0999 & \textbf{443} & 3.47 &  &  &  &  \\ 
 & & $P$=2, $b$=20 & 441 & 5.76 & 0.0999 & 443 & 7.26 &  &  &  &  \\ \midrule
\multirow{2}{*}{AMiner} &  \multirow{2}{*}{0.5} & $P$=1, $b$=50 & 2440 & 108 & 0.4999 & 2,449 & 143 & \multirow{2}{*}{\textbf{2,115}} & \multirow{2}{*}{1,049} & \multirow{2}{*}{8,018} & \multirow{2}{*}{399}  \\
 & & $P$=2, $b$=50 & \textbf{2,229} & 134 & 0.4999 & \textbf{2,242} & 205 &   &  &  &  \\ 
\bottomrule 
 \end{tabular}
 \end{spacing}
 }
\end{table} 
 \begin{figure*}[h]
\centering
\subfigure[original image] {\includegraphics[height=1.7in]{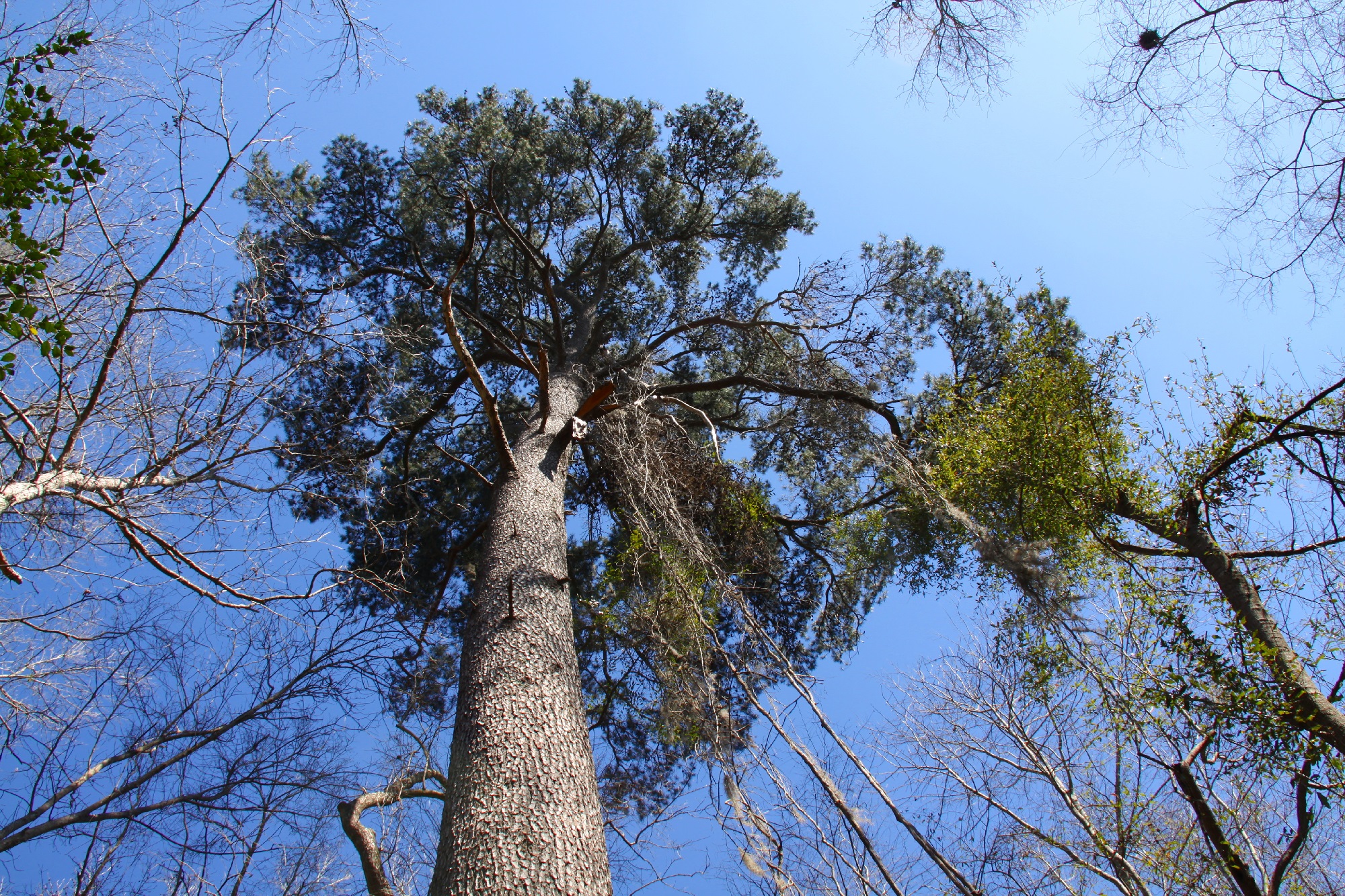}}
\hspace{-5pt}
\subfigure[compressed image (10\% error)] {\includegraphics[height=1.7in]{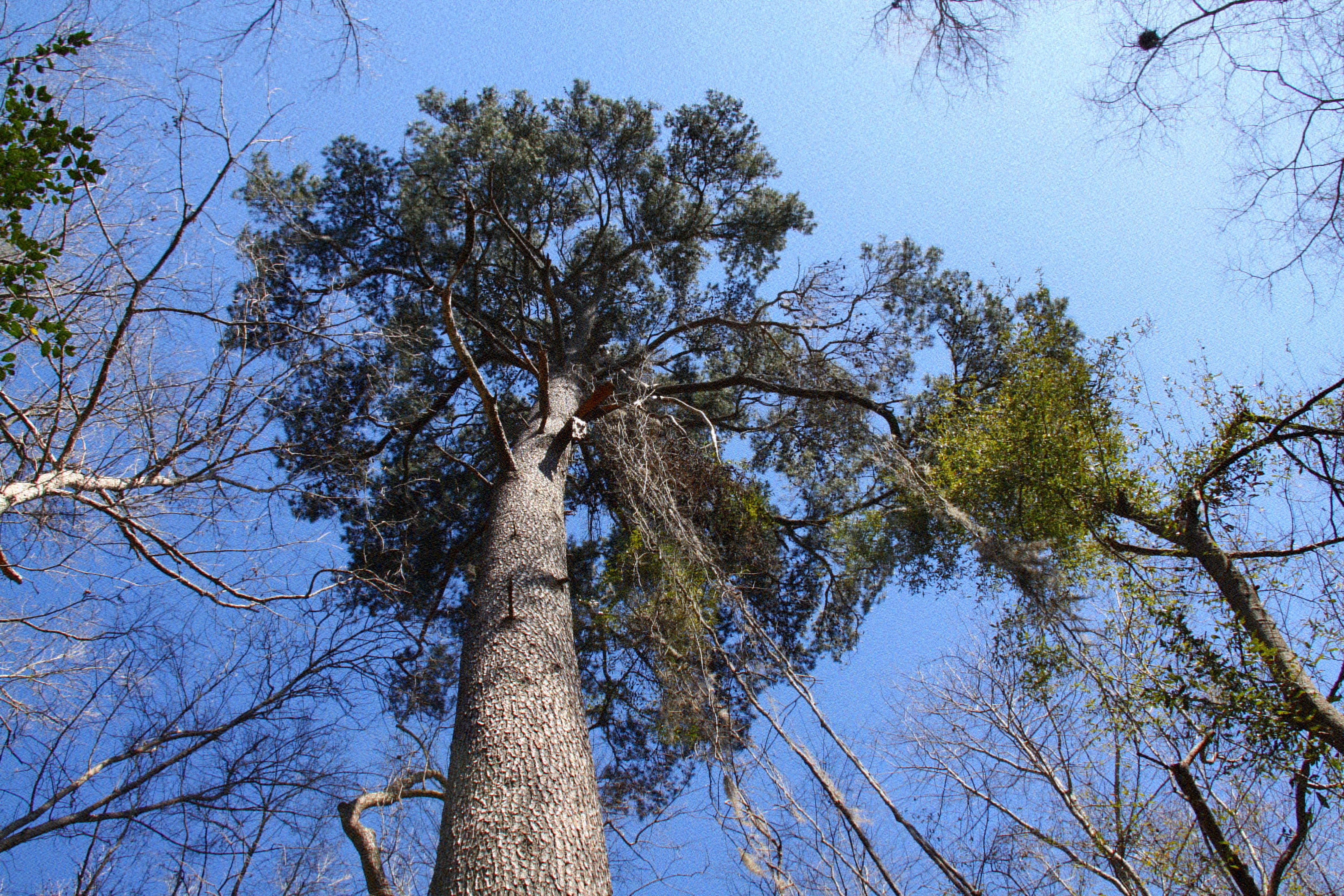}}
\caption{The original image and the compressed image obtained with the proposed algorithm.}
\label {image_comp}
\end{figure*}

   For the second data, which is a large sparse matrix, the Krylov subspace iterative method ``\texttt{svds}'' \cite{SaadBook, Matbook2007} is also tested. However, it costs 2,281 seconds for computing the first 1,000 singular values/vectors. It is much slower than executing ``\texttt{svd}'' to the matrix's dense version. Besides, ``\texttt{svd}'' requests more than 20 GB memory, while the proposed randomized algorithms only costs 3 GB memory or so for this case.

\section{Conclusions}
Efficient techniques are proposed for the fixed-precision low-rank approximation of large matrices. Our contributions are as follows.
\begin{itemize}%[leftmargin=*]
\item A simple and accurate error indicator in Frobenius norm is proposed, which enables efficient rank determination and can be used in the blocked randQB algorithm \cite{Martin2015} and other incremental QB-form factorization algorithms (like that in \cite{Duersch2015}). We have proved its accuracy and validity for the problems with relative accuracy tolerance larger than $2.1\times 10^{-7}$. Numerical experiments on large dense and sparse matrices have shown that the proposed rank determination scheme brings several to several tens times speedup and memory saving to the blocked randQB algorithm, without loss of accuracy. 
\item Base on the blocked randQB algorithm, we propose a pass-efficient algorithm called {\randQBfp}. It is mathematically equivalent to the blocked randQB algorithm, but reduces the passes over matrix $\bA$ to the fewest. The {\randQBfp} algorithm also suits to the fixed-precision problem, and can derive a single-pass algorithm under certain condition. Numerical results have validated the efficiency and accuracy of the {\randQBfp} algorithm, and shown that the derived single-pass algorithm is much more accurate than an existing counterpart.   
\item Real data are tested to demonstrate the effectiveness of the proposed algorithms for the fixed-precision problem. Compared with the adaptive range finder approach \cite{Halko2011}, the proposed algorithms run faster and produce much smaller factor matrices while attaining the accuracy criterion.  
\end{itemize}

Future work includes extending and applying the proposed algorithms to more practical data mining and machine learning scenarios.

%The proposed techniques are suitable for large and/or sparse matrices, especially for solving the fixed-precision problem. Although our presentation is mainly based on the randomized QB factorization \cite{Martin2015, Halko2011}, the proposed error estimation mechanism can be applied to the BLAS-3 QRCP algorithm \cite{Duersch2015}, which might attract more practical applications of low-rank approximation.
%Another potential application can be the tensor network SVD \cite{tensorSVD}, for which the proposed technique is ready to accelerate the $\delta$-truncated SVD computation.
%For the {\randQBfp} algorithm, its application to problems with streaming data and its enhancement with adaptive setting of parameters $\tilde{l}$ and $b$ will be explored in the future. 

%\section*{Acknowledgments}
%The authors thank Prof. P.-G. Martinsson for sharing the source code of the blocked randQB algorithm, which makes this work possible. 

\end{document}